\documentclass[a4paper]{amsart}
\usepackage{amsmath}
\usepackage{amssymb}
\usepackage{amsthm}
\usepackage{enumerate}
\usepackage{tikz}
\usetikzlibrary{decorations.pathreplacing, calc}
\newtheorem{theorem}{Theorem}
\newtheorem*{theorem*}{Theorem}
\newtheorem{lemma}{Lemma}
\newtheorem*{lemma*}{Lemma}
\newtheorem{corollary}{Corollary}
\newtheorem*{corollary*}{Corollary}
\newtheorem*{example*}{Example}
\newtheorem{proposition}{Proposition}
\newtheorem*{proposition*}{Proposition}

\theoremstyle{definition}
\newtheorem{definition}{Definition}
\newtheorem*{definition*}{Definition}
\newtheorem{remark}{Remark}
\newtheorem*{remark*}{Remark}

\usepackage{multirow}
\allowdisplaybreaks[1]
\def\MAGNIFICATION{1.1}

\def\POINTSIZE{4 * \MAGNIFICATION}

\def\LINEWIDTH{1 * \MAGNIFICATION}

\def\INDEX_SIZE{\footnotesize }

\def\DOT_DISTANCE{0.8 * \MAGNIFICATION}
\def\LINE_DISTANCE{0.5 * \MAGNIFICATION}
\def\BRANCH_WIDTH{1 * \MAGNIFICATION}
\def\SECOND_EDGE_PROPOTION{0.2}

\def\BRACEHSPACE{0.08 * \MAGNIFICATION}
\def\BRACEVSPACE{0.11 * \MAGNIFICATION}
\def\BRACEAMPLITUDE{2 mm * \MAGNIFICATION}
\def\BRACEXSIFT{6.5 mm * \MAGNIFICATION}

\def\VDOTSIZE{\LINEWIDTH}
\def\FROMVDOT{0.3}
\def\DOTSIZE{1 * \MAGNIFICATION}
\def\DOTSBRANK{0.6 * \MAGNIFICATION}
\def\DOTSPOSITION{0 * \MAGNIFICATION}

\def\drawvdots#1#2{
\coordinate (C#1_#2_0) at ($(#1)!\FROMVDOT!(#2)$);
\coordinate (C#1_#2_1) at ($(#1)!{1-\FROMVDOT}!(#2)$);
\draw[line width = \LINEWIDTH] (#1) -- (C#1_#2_0) (#2) -- (C#1_#2_1);
\tvdot (D#1_#2_0) at ($(C#1_#2_0)!.25!(C#1_#2_1)$) {};
\tvdot (D#1_#2_1) at ($(C#1_#2_0)!.5!(C#1_#2_1)$) {};
\tvdot (D#1_#2_2) at ($(C#1_#2_0)!.75!(C#1_#2_1)$) {};
}

\def\Blackcircle{\node [circle, draw, black, fill=black,
inner sep=0pt, minimum width=\POINTSIZE, label distance=1mm]}
\def\Whitecircle{\node [circle, draw, black, fill=white,
inner sep=0pt, minimum width=\POINTSIZE, label distance=1mm]}
\def\tdot{\node [circle, draw, black, fill=black,
inner sep=0pt, minimum width=\DOTSIZE, label distance=1mm]}
\def\tvdot{\node [circle, draw, black, fill=black,
inner sep=0pt, minimum width=\VDOTSIZE, label distance=1mm]}

\begin{document}
\title[On an analog of the Arakawa-Kaneko zeta function]{On an analog of the Arakawa-Kaneko zeta function and relations of some multiple zeta values}
\author{RYOTA UMEZAWA}
\date{}
\maketitle
\begin{abstract}
T. Ito defined an analog of the Arakawa-Kaneko zeta function to obtain relations among Mordell-Tornheim multiple zeta values. In this paper, we develop two things related to an analog of the Arakawa-Kaneko zeta function. One is to find an analog of the Arakawa-Kaneko zeta function of Miyagawa-type (defined by T. Miyagawa) and to obtain a relation among  Miyagawa-type MZVs. The other is to find a class of  zeta functions to which Ito's zeta functions of  the case of general index are related.
\end{abstract}
\section{Introduction}
The Arakawa-Kaneko zeta function is the following function introduced in \cite{ArakawaKaneko}.
\begin{definition}[The Arakawa-Kaneko zeta function]
For $\mathbf{k} = (k_{1}, \dots, k_{n}) \in \mathbb{N}^{n}$ and $s \in \mathbb{C}$ with $\Re(s) > 1-n$, the Arakawa-Kaneko zeta function is defined by
\[\xi(\mathbf{k};s)= \frac{1}{\Gamma(s)} \int_{0}^{\infty} t^{s-1} \frac{{\rm Li}_{\mathbf{k}}(1-e^{-t})}{e^{t}-1}\, dt,\]
where ${\rm Li}_{\mathbf{k}}(z)$ is the multi-polylogarithm defined by
\[{\rm Li}_{\mathbf{k}}(z) = \sum_{0<m_{1}<m_{2} < \dots < m_{n}} \frac{z^{m_{n}}}{m_{1}^{k_{1}} m_{2}^{k_{2}} \cdots m_{n}^{k_{n}}}\qquad (|z|<1).\]
\end{definition}

The Arakawa-Kaneko zeta function has a connection with Euler-Zagier multiple zeta values (EZ-type MZVs for brevity) and poly-Bernoulli numbers (see \cite{ArakawaKaneko}). Here, EZ-type MZVs are the special values of the following functions. 
\begin{definition}[The Euler-Zagier multiple zeta function (EZ-type MZF)]
For $\mathbf{s} = (s_{1},\dots,s_{n}) \in \mathbb{C}^{n}$, the Euler-Zagier multiple zeta function is defined by
\[\zeta(\mathbf{s}) = \sum_{0 < m_1 < m_2 < \dots < m_n} \frac{1}{m_{1}^{s_{1}} m_{1}^{s_{1}} \cdots m_{n}^{s_{n}}}.\]
\end{definition}
This series converges absolutely when
\[\sum_{i=0}^{k} \Re (s_{n-i}) > k+1\]
for any $k$ with $0 \le k \le n-1$ (see \cite{Matsumoto2}) and can be continued meromorphically to the whole $\mathbb{C}^{n}$ space (see \cite{AkiyamaEgamiTanigawa}). The values of EZ-type MZFs at $\mathbf{k} = (k_{1}, \dots, k_{n}) \in \mathbb{N}^{n}$ with $k_{n} \ge 2$ are called Euler-Zagier multiple zeta values (EZ-type MZVs).

In this paper, we focus on the fact that the properties of  the Arakawa-Kaneko zeta function lead to certain relations among EZ-type MZVs. Regarding this, there is the work of Ito \cite{Ito}. Ito introduced the following function as an analog of the Arakawa-Kaneko zeta function:
\begin{definition}[The Ito zeta function]
For $k_{1}, \dots, k_{r} \in \mathbb{N}$ and $s \in \mathbb{C}$ with $\Re(s) > 1-r$, we define
\begin{equation}\label{eq:Ito}
\xi_{MT}(k_{1},\dots,{k_{r}};s) = \frac{1}{\Gamma(s)} \int_{0}^{\infty} t^{s-1} \frac{\prod_{i=1}^{r} {\rm Li}_{k_{i}}(1-e^{-t})}{e^{t}-1} \,dt.
\end{equation}
\end{definition}
We call the function (\ref{eq:Ito}) the Ito zeta function in this paper. We also introduce MT-type MZVs which are the special values of the following functions.
\begin{definition}[The Mordell-Tornheim multiple zeta function (MT-type MZF)]
For $s_{1},\dots,s_{r+1}\in\mathbb{C}$, the Mordell-Tornheim multiple zeta function is defined by
\[\zeta_{MT}(s_{1}, \dots, s_{r}; s_{r+1}) = \sum_{m_{1}=1, \dots, m_{r}=1}^{\infty} \frac{1}{m_{1}^{s_{1}} \cdots m_{r}^{s_{r}} (m_{1} + \dots + m_{r})^{s_{r+1}}}.\]
\end{definition}
This series converges absolutely when
\[\sum_{l=1}^{j} \Re(s_{k_{l}}) + \Re(s_{r+1}) > j\]
with $1 \le k_{1} < k_{2} < \dots < k_{j} \le r$ for any $j=1,2,\dots,r$ (see \cite{OkamotoOnozuka}) and can be continued meromorphically to the whole $\mathbb{C}^{n}$ space (see \cite{Matsumoto}). The values of MT-type MZF at non-negative integer points in the domain of convergence are called Mordell-Tornheim multiple zeta values (MT-type MZVs). Ito used his zeta function to obtain certain relations among MT-type MZVs. Therefore, Ito zeta function is an analog of the Arakawa-Kaneko zeta function of MT-type.

There is a generalization of Ito zeta function, which was given by Ito himself as follows.
\begin{definition}[The Generalized Ito zeta function ($r=1$)]
For $\mathbf{k} = (k_1, \dots, k_n)\in\mathbb{N}^{n}$, $k_{n+1} \in \mathbb{Z}_{\ge0}$ and $s\in \mathbb{C}$ with $\Re(s) >1-n$ we define
\begin{equation}\label{eq:AKZofM}
\xi((\mathbf{k}; k_{n+1});s) = \frac{1}{\Gamma(s)} \int_{0}^{\infty} \frac{t^{s-1}}{e^{t}-1} {\rm Li}_{\mathbf{k}; k_{n+1}}(1-e^{-t})\,dt,
\end{equation}
where
\[{\rm Li}_{\mathbf{k}; k_{n+1}}(z) = \sum_{m_{1}=1,\dots,m_{n}=1}^{\infty} \frac{z^{\sum_{j=1}^{n} m_{j}}}{m_{1}^{k_{1}} \cdots m_{n}^{k_{n}} (\sum_{j=1}^{n} m_{j})^{k_{n+1}}}\quad (|z|<1).\]
\end{definition}
Ito considered a version of the function (\ref{eq:AKZofM}), in which ${\rm Li}_{\mathbf{k};k_{n+1}}(1-e^{-t})$ is replaced by a product of  $r$ quantities of the form ${\rm Li}_{\mathbf{k};k_{n+1}}(1-e^{-t})$ (\cite[Definition 13]{Ito}). Therefore, we call the function (\ref{eq:AKZofM}) generalized Ito zeta function ($r=1$).

On the other hand, There is also a generalization of MT-type MZF, which was given by Miyagawa as follows.
\begin{definition}[The Miyagawa multiple zeta function (Miyagawa-type MZF)]
For $s_{1},\dots,s_{r+1}\in\mathbb{C}$, we define
\begin{align}
&\hat{\zeta}_{MT,j,r} (s_{1},\dots,s_{j};s_{j+1},\dots,s_{r+1}) \label{de:MiyagawaMZF}\\
&=\sum_{m_{1}=1,\dots,m_{r}=1}^{\infty} \frac{1}{m_{1}^{s_{1}} \cdots m_{j}^{s_{j}} \prod_{u=j+1}^{r+1} (\sum_{v=1}^{u-1} m_{v} )^{s_{u}}}. \nonumber
\end{align}
\end{definition}
This function was introduced by Miyagawa \cite{Miyagawa}. Moreover, he showed that the function (\ref{de:MiyagawaMZF}) can be continued meromorphically
to the whole $\mathbb{C}^{r+1}$ space.
We call the function (\ref{de:MiyagawaMZF}) the Miyagawa multiple zeta function (Miyagawa-type MZF). Moreover, we call the values of Miyagawa-type MZFs at non-negative integer points in the domain of convergence Miyagawa multiple zeta values (Miyagawa-type MZVs).
In the present paper, we use the function (\ref{eq:AKZofM}) to obtain certain relations among Miyagawa-type MZVs.
Ito's method uses functional relations between functions (\ref{eq:Ito}) and MT-type MZFs (Theorem \ref{th:IZfrel}), and our method uses functional relations between functions (\ref{eq:AKZofM}) and Miyagawa-type MZFs (Theorem \ref{th:AKZofMfre}). However, Theorem \ref{th:IZfrel} and Theorem \ref{th:AKZofMfre} only give functional relations in the case when all the indices $k_i$ are 2 in the function (\ref{eq:Ito}) and the function (\ref{eq:AKZofM}). In this paper, we study also functional relations for the function (\ref{eq:Ito}) and the function (\ref{eq:AKZofM}) with general indices. For this purpose, we introduce the following new class of multiple zeta functions.
\begin{definition}[The generalized Mordell-Tornheim multiple zeta function (GMT-type MZF)]\label{de:Uzeta}
For $\mathbf{s}_{i}=(s_{i,1},\dots,s_{i,n_{i}}) \in \mathbb{C}^{n_{i}}\, (1 \le i \le r+1,\,n_{i} \in \mathbb{N})$, we define
\begin{align}
&\zeta_{MT}(\mathbf{s}_{1},\mathbf{s}_{2},\dots,\mathbf{s}_{r};\mathbf{s}_{r+1})\label{eq:defU}\\
&=\sum_{\substack{0<m_{1,1}<m_{1,2}<\dots<m_{1,n_{1}} \\ \vdots \\ 0<m_{r,1}<m_{r,2}<\dots<m_{r,n_{r}}}} \, \sum_{m_{r+1,1}=1,\dots ,m_{r+1,n_{r+1}-1}=1}^{\infty} \nonumber \\
&\quad\frac{1}{\prod_{i=1}^{r} \prod_{j=1}^{n_{i}} m_{i,j}^{s_{i,j}} \prod_{u=1}^{n_{r+1}} (\sum_{v=1}^{r} m_{v,n_{v}} +\sum_{w=1}^{u-1} m_{r+1,w})^{s_{r+1,u}}} \nonumber \\
&=\sum_{\substack{0<m_{1,1},0<m_{1,2}, \dots, 0<m_{1,n_{1}} \\ \vdots \\ 0<m_{r,1},0<m_{r,2},\dots, 0<m_{r,n_{r}}}} \, \sum_{m_{r+1,1}=1,\dots ,m_{r+1,n_{r+1}-1}=1}^{\infty} \nonumber \\
&\quad\frac{1}{\prod_{i=1}^{r} \prod_{k=1}^{n_{i}}(\sum_{j=1}^{k} m_{i,j})^{s_{i,k}} \prod_{u=1}^{n_{r+1}} (\sum_{i=1}^{r} \sum_{j=1}^{n_{i}} m_{i,j}+ \sum_{w=1}^{u-1} m_{r+1,w})^{s_{r+1,u}}} \nonumber.
\end{align}
\end{definition}
We call this function the generalized Mordell-Tornheim multiple zeta function (GMT-type MZF) and call the values of GMT-type MZFs at non-negative integer points in the domain of convergence generalized Mordell-Tornheim multiple zeta values (GMT-type MZVs).
As a consequence of the present study, we can obtain relations among special values of functions (\ref{eq:defU}).

Regarding relations among those functions, the known results and the results shown in the present paper are summarized as follows.
\begin{table}[hbtp]
\renewcommand{\arraystretch}{1.2}
\label{table:data_type}
\centering
\begin{tabular}{c||c|cc}
\hline
$\xi$-function                                                    & special value  & \multicolumn{2}{|c}{functional relation} \\ \hline \hline
Arakawa-Kaneko & EZ-type & EZ-type &\\
zeta function& (Theorem \ref{th:AKZspv}) & (Theorem \ref{th:AKZfrel}) & \\ \hline
& & MT-type  & \multirow{2}{*}{if $k_{i} \le 2$} \\ 
Ito & MT-type & (Theorem \ref{th:ItoGfre}) & \\ \cline{3-4}
zeta function & (Theorem \ref{th:IZspv}) & GMT-type ($n_{r+1} = 1$) & \multirow{2}{*}{general} \\
 & & (Theorem \ref{th:ItoGfre}) & \\ \hline
\multirow{2}{*}{Generalized Ito} & & Miyagawa-type & if $k_{i} \le 2$ \\
\multirow{2}{*}{zeta function}& Miyagawa-type & (Theorem \ref{th:AKZofMGfre}) &($1\le i \le n$) \\ \cline{3-4}
\multirow{2}{*}{($r=1$)} & (Theorem \ref{th:AKZofMspv}) & GMT-type &\multirow{2}{*}{general} \\
 &   &(Theorem \ref{th:AKZofMGfre})&\\ \hline
\hline
\end{tabular}
\end{table}
\begin{remark*} Theorem \ref{th:IZfrel} expresses the relationship between Ito zeta functions with $k_{i} = 2$ and MT-type MZFs, and  Theorem \ref{th:AKZofMfre} expresses the relationship between generalized Ito zeta functions $(r=1)$ with $k_{i} = 2\ (1 \le i \le n)$ and Miyagawa-type MZFs.
\end{remark*}

In Section \ref{se:Pre}, we provide some notations, lemmas and known results which we need in later sections. In Section \ref{se:MtypeAK}, we discuss the function (\ref{eq:AKZofM}) and prove Theorem \ref{th:AKZofMspv} and Theorem \ref{th:AKZofMfre}. As a consequence, we can obtain relations among Miyagawa-type MZVs. In Section 
\ref{se:Uzeta}, we discuss the function (\ref{eq:defU}) and prove several propositions on the function (\ref{eq:defU}). In Section \ref{se:Generalize}, using  functions (\ref{eq:defU}), we generalize Theorem \ref{th:IZfrel} and \ref{th:AKZofMfre} to Theorem \ref{th:ItoGfre} and \ref{th:AKZofMGfre}, respectively. As a consequence, we can obtain relations among GMT-type MZVs.
\section{Preliminaries}\label{se:Pre}
In this section, we provide some notations, lemmas and known results which we need later.
In this paper,  unless otherwise noted, $k,\,n$ and $r$ denote positive integers and $s$ is a complex number also when these have subscripts. Moreover, $\{k\}^{n}$ denotes $n$ repetitions of $k$. For example, $(1,2,2,3) = (1,\{2\}^{2},3)$. For $\mathbf{k} = (k_{1}, \dots, k_{n})$, we define $\mathbf{k}_{\pm} = (k_{1}, \dots, k_{n-1},k_{n}\pm1)$ and ${}_{\pm}\mathbf{k} = (k_{1}\pm1,k_{2}, \dots, k_{n})$. Let $\mathbf{k}^{*}$ denote the  dual index of $\mathbf{k}$.

\begin{lemma}[{\cite[Lemma 1]{ArakawaKaneko} and \cite[Lemma 4]{Ito}}]\label{le:Lidi}
For $\mathbf{k} = (k_{1}, \dots, k_{n}) \in \mathbb{N}^{n}$, the following formulas hold:
\begin{enumerate}[{\rm (i)}]
\item\[\frac{d}{dz}{\rm Li}_{\mathbf{k}}(z) = \begin{cases}
\frac{1}{z} {\rm Li}_{\mathbf{k}_{-}}(z)&({\rm if}\ k_{n} > 1),\\
 \frac{1}{1 - z} {\rm Li}_{k_1, \dots ,k_{n-1}}(z)&({\rm if}\ k_{n} = 1).
\end{cases}\]
\item\[\frac{d}{dz}{\rm {Li}}_{\mathbf{k}; k_{n+1}}(z) = \begin{cases}
\frac{1}{z} {\rm Li}_{\mathbf{k}; k_{n+1} - 1}(z)&({\rm if}\ k_{n+1} > 1),\\
\frac{1}{z} \prod_{i=1}^{n}{\rm Li}_{k_i}(z)&({\rm if}\ k_{n+1} =1).
\end{cases}\]
\end{enumerate}
\end{lemma}
By Lemma \ref{le:Lidi}, we have
\[\frac{d}{du}{\rm Li}_{\mathbf{k}_{+}}(1-e^{-t-u}) = \frac{1}{e^{t+u}-1}{\rm Li}_{\mathbf{k}}(1-e^{-t-u}),\]
\[\frac{d}{du}{\rm Li}_{\mathbf{k}; 1}(1-e^{-t-u}) = \frac{1}{e^{t + u} - 1} \prod_{i=1}^{n}{\rm Li}_{k_i}(1-e^{-t - u}),\]
\[\frac{d}{du}{\rm Li}_{\mathbf{k};k_{n+1}+1}(1-e^{-t-u}) = \frac{1}{e^{t + u} - 1} {\rm Li}_{\mathbf{k};k_{n+1}}(1-e^{-t - u}).\]
Therefore, we obtain the following corollary.
\begin{corollary}\label{co:liint}
The following formulas hold:
\[\int_{0}^{\infty} \frac{1}{e^{t + u} - 1} {\rm Li}_{\mathbf{k}}(1-e^{-t - u})\, du = \zeta(\mathbf{k}_{+}) - {\rm Li}_{\mathbf{k}_{+}} (1 - e ^{-t}), \]
\[\int_{0}^{\infty} \frac{1}{e^{t + u} - 1} \prod_{i=1}^{n}{\rm Li}_{k_i}(1-e^{-t - u})\, du = \zeta_{MT}(\mathbf{k}; 1) - {\rm Li}_{\mathbf{k}; 1} (1 - e ^{-t}), \]
\[\int_{0}^{\infty} \frac{1}{e^{t + u} - 1} {\rm Li}_{\mathbf{k};k_{n+1}}(1-e^{-t - u})\, du = \zeta_{MT}(\mathbf{k}; k_{n+1} + 1) - {\rm Li}_{\mathbf{k}; k_{n+1} + 1} (1 - e ^{-t}).\]
\end{corollary}
\begin{lemma}\label{le:Aintrep}
For a matrix $A=(a_{i,j})_{1 \le i \le n, 1 \le j \le r},\ a_{i,j} \in \mathbb{R}_{\ge 0}$ and $\mathbf{s} = (s_{1}, \dots, s_{n}) \in \mathbb{C}^{n}$ we define
\begin{align}
&\zeta(\mathbf{s};A) \label{eq:Mzeta}\\
&=\sum_{m_{1}=1,\dots,m_{r}=1}^{\infty} (a_{1,1}m_{1}+ \dots+ a_{1,r}m_{r})^{-s_{1}} \cdots (a_{n,1}m_{1}+ \dots +a_{n,r}m_{r})^{-s_{n}} \nonumber \\
&=\sum_{m_{1}=1,\dots,m_{r}=1}^{\infty}\prod_{i=1}^{n}\frac{1}{(A
\begin{pmatrix}
m_{1} \\
\vdots \\
m_{r}
\end{pmatrix}
)_{i}^{s_{i}}} \nonumber
\end{align}
where $\,()_{i}$ represents the i-th element of a vertical vector. 
For $\Re(s_{i}) > 0\ (1 \le i \le n)$, if $\zeta(\mathbf{s};A)$ converges absolutely then 
\begin{equation}\label{eq:Aintrep}
\zeta(\mathbf{s};A)=\frac{1}{\prod_{i=1}^{n}\Gamma(s_{i})}\int_{0}^{\infty} \cdots \int_{0}^{\infty} \Biggl(\prod_{i=1}^{r}\frac{1}{ \exp{({}^{t}\!A
\begin{pmatrix}
t_{1} \\
\vdots \\
t_{n}
\end{pmatrix}
)_{i}}-1}\Biggr)\prod_{i=1}^{n}t_{i}^{s_{i}-1}\,dt_{i}.
\end{equation}
where ${}^{t}\!A$ represents the transposed matrix of A.
\end{lemma}
\begin{remark}\label{re:matrixnotinuas}
The function $\zeta(\mathbf{s};A)$ was introduced by Matsumoto \cite{Matsumoto}. Moreover, he showed that the function $\zeta(\mathbf{s};A)$ can be continued meromorphically
to the whole $\mathbb{C}^{n}$ space.
\end{remark}
\begin{remark}
If there exists $i$ satisfying $\Re(s_{i}) \le 0$, then the right hand side of (\ref{eq:Aintrep}) diverges to the infinity.
\end{remark}
\begin{proof}[Proof of Lemma {\ref{le:Aintrep}}]
By using $m^{-s} = \frac{1}{\Gamma(s)}\int_{0}^{\infty}t^{s-1} e^{-mt}\,dt\ (\Re(s) > 0)$, we have
\begin{align*}
&\zeta(\mathbf{s};A) \\
&=\sum_{m_{1}=1,\dots,m_{r}=1}^{\infty} \frac{1}{\prod_{i=1}^{n}\Gamma(s_{i})} \\
&\quad\times \int_{0}^{\infty}\!\cdots\int_{0}^{\infty} e^{-t_{1}(a_{1,1}m_{1}+\dots+a_{1,r}m_{r})}  \cdots e^{-t_{n}(a_{n,1}m_{1}+\dots+a_{n,r}m_{r})}\prod_{i=1}^{n}t_{i}^{s_{i}-1}\,dt_{i}\\
&=\sum_{m_{1}=1,\dots,m_{r}=1}^{\infty} \frac{1}{\prod_{i=1}^{n}\Gamma(s_{i})} \\
&\quad\times \int_{0}^{\infty}\cdots\int_{0}^{\infty} e^{-(t_{1}a_{1,1}+\dots+t_{n}a_{n,1})m_{1}}\cdots e^{-(t_{1}a_{1,r}+\dots+t_{n}a_{n,r})m_{r}}\prod_{i=1}^{n}t_{i}^{s_{i}-1}\,dt_{i}\\
&=\frac{1}{\prod_{i=1}^{n}\Gamma(s_{i})}\int_{0}^{\infty}\!\cdots\!\int_{0}^{\infty}
\frac{1}{e^{t_{1}a_{1,1}+\dots+t_{n}a_{n,1}}-1} \cdots \frac{1}{e^{t_{1}a_{1,r}+\dots+t_{n}a_{n,r}}-1}\prod_{i=1}^{n}t_{i}^{s_{i}-1}\,dt_{i}.
\end{align*}
The last equality holds since $\zeta(\mathbf{s};A)$ converges absolutely.
\end{proof}
The following results on the Arakawa-Kaneko zeta function and the Ito zeta function are known.
\begin{theorem}[{\cite[Theorem 2.5]{KanekoTsumura}}]\label{th:AKZspv}
For $\mathbf{k} = (k_{1}, \dots, k_{n}) \in \mathbb{N}^{n}$, we write $|\mathbf{k}|=k_{1}+\dots+k_{n}$ and call it the weight of $\mathbf{k}$, and $d(\mathbf{k})=n$, the depth of $\mathbf{k}$. Moreover, we write
\[b(\mathbf{k}; \mathbf{j}) = \prod_{i=1}^{n} \binom{k_{i} + j_{i} - 1}{j_{i}}.\]For any $m \in \mathbb{N}$, we have
\[\xi (\mathbf{k}; m) = \sum_{|\mathbf{j}| = m-1,\,d(\mathbf{j}) = d((\mathbf{k}_{+})^{*})} b((\mathbf{k}_{+})^{*} ; \mathbf{j}) \zeta((\mathbf{k}_{+})^{*} + \mathbf{j}),\]
where the  sum is over all $\mathbf{j}=(j_{1},\dots,j_{n}) \in \mathbb{Z}_{\ge 0}^{n}$ satisfying $|\mathbf{j}| = m-1,\,d(\mathbf{j}) = d((\mathbf{k}_{+})^{*})$.
\end{theorem}
\begin{theorem}[{\cite[Theorem 3.6]{KanekoTsumura}}]\label{th:AKZfrel}
The Arakawa-Kankeo zeta function $\xi (\mathbf{k}; s)$ can be written in terms of EZ-type MZFs as
\[\xi (\mathbf{k}; s) = \sum_{\mathbf{k'},\,j \ge 0} c_{\mathbf{k}} (\mathbf{k'}; j) \binom{s+j-1}{j} \zeta(\mathbf{k'}, j+s).\]
Here, the sum is over indices $\mathbf{k'}$ and integers $j \ge 0$ satisfying $|\mathbf{k'}| + j \le |\mathbf{k}|$, and $c_{\mathbf{k}} (\mathbf{k'}; j)$ is a $\mathbb{Q}$-linear combination of EZ-type MZVs of weight $|\mathbf{k}| - |\mathbf{k'}| - j$.
\end{theorem}

\begin{theorem}[{\cite[Proposition 5]{Ito}}]\label{th:IZspv}
For $m \in \mathbb{Z}_{\ge 0}$,
\[\xi_{MT} (k_{1},\dots,k_{r}; m+1) = \frac{1}{m!} \zeta_{MT}(k_{1}, \dots, k_{r}, \{1\}^{m}; 1).\]
\end{theorem}

\begin{theorem}[{\cite[Theorem 8]{Ito}}]\label{th:IZfrel}
For $r \in \mathbb{N}$ and $s \in \mathbb{C}$,
\begin{align*}
&\sum_{j=0}^{r-1} \binom{r-1}{j} (-1)^{j} \zeta(2)^{r-1-j} \Gamma(s) \xi_{MT} (\{2\}^{j}; s)\\
&=\sum_{j=0}^{r-1} \binom{r-1}{j} \Gamma(s+j) \zeta_{MT}(0, \{2\}^{r-1-j}, \{1\}^{j}; j+s).
\end{align*}
\end{theorem}
\begin{remark}\label{re:obtainrel}
Ito obtained relations among MT-type MZVs by putting $s=m+1$ in Theorem \ref{th:IZfrel} and using Theorem \ref{th:IZspv} for $\xi_{MT}(\{2\}^{j};m+1)$.
\end{remark}

\section{On an analog of the Arakawa-Kaneko zeta function of Miyagawa-type}\label{se:MtypeAK}
Miyagawa \cite{Miyagawa} defined the multiple zeta funcion (\ref{de:MiyagawaMZF}). We write the Miyagawa-type MZF as follows.
\begin{definition}
For $\mathbf{s}_{r+1} = (s_{r+1,1}, \dots , s_{r+1,n_{r+1}}) \in \mathbb{C}^{n_{r+1}}$, we write
\begin{align*}
&\zeta_{MT} (s_{1},\dots,s_{r};\mathbf{s}_{r+1}) \\
&=\sum_{0<m_{1},\dots,0<m_{r}} \sum_{m_{r+1,1}=1, \dots, m_{r+1,n_{r+1}-1}=1}^{\infty} \\
&\quad\frac{1}{m_{1}^{s_{1}} \cdots m_{r}^{s_{r}} \prod_{u=1}^{n_{r+1}} (\sum_{v=1}^{r} m_{v} + \sum_{w=1}^{u-1}m_{r+1,w})^{s_{r+1,u}}}.
\end{align*}
\end{definition}

\begin{proposition}
The Miyagawa-type MZF $\zeta_{MT}(s_{1},s_{2},\dots,s_{r}; \mathbf{s}_{r+1})$ converges absolutely when
\[\sum_{i=0}^{k} \Re(s_{r+1,n_{r} -i}) > k + 1\]
for any $k=1,\dots,n_{r} - 2$ and
\[\sum_{l=1}^{j}\Re(s_{k_l}) + \sum_{i=0}^{n_{r+1} -1} \Re(s_{r+1,n_{r+1} -i})-n_{r+1}+1 > j\]
with $1 \le k_{1} < k_{2} < \dots < k_{j} \le r$ for any $j=1,2,\dots,r$ are all satisfied.
\end{proposition}

\begin{proof}
The series $\sum_{n=1}^{\infty} \frac{1}{(N+n)^\sigma}\ (N>0)$ converges only when $\sigma>1$, and
\begin{equation*}
\sum_{n=1}^{\infty} \frac{1}{(N+n)^\sigma} \le \frac{1}{(\sigma - 1)N^{\sigma - 1}}.
\end{equation*}
Let
\[\mathbf{s}^{(k)} = (\Re(s_{r+1,1}), \dots  ,  \Re(s_{r+1,n_{r+1}-k-1}), \sum_{i=0}^{k}\Re(s_{r+1,n_{r+1}-i} )  - k).\]
Then we have
\begin{align*}
&\zeta_{MT}(\Re(s_{1}),\Re(s_{2}),\dots,\Re(s_{r}); \Re(s_{r+1,1}), \dots, \Re(s_{r+1,n_{r+1}})) \\
&\ll \zeta_{MT}(\Re(s_{1}),\Re(s_{2}),\dots,\Re(s_{r}); \mathbf{s}^{(1)}) \\
&\ll \zeta_{MT}(\Re(s_{1}),\Re(s_{2}),\dots,\Re(s_{r}); \mathbf{s}^{(2)}) \\
&\ll\cdots\\
&\ll \zeta_{MT}(\Re(s_{1}),\Re(s_{2}),\dots,\Re(s_{r});\sum_{i=0}^{n_{r+1}-1} \Re(s_{r+1,n_{r+1} -i}) -n_{r+1} + 1),
\end{align*}
where the implicit constants of $\ll$ depend on $\mathbf{s}_{r+1} =(s_{r+1,1}, \dots, s_{r+1,n_{r+1}})$.
By absolute convergence of the MT-type MZF, the assertion of this proposition follows.
\end{proof}
\begin{remark}
By Lemma \ref{le:Aintrep}, we have
\begin{align}
&\zeta_{MT}(0, s_{2},\dots, s_{r};\mathbf{s}_{r+1}) \label{eq:Mintrep} \\
&=\frac{1}{\prod_{i=2}^{r} \Gamma(s_{i})\prod_{j=1}^{n_{r+1}}\Gamma(s_{r+1,j})}\int_{0}^{\infty}\cdots\int_{0}^{\infty} \prod_{i=2}^{r} \frac{t_{i}^{s_{i}-1}\,dt_{i}}{e^{t_{i} + t_{r+1,1} + \dots +t_{r+1,n_{r+1}}} -1} \nonumber \\
&\quad\times\,\prod_{u=1}^{n_{r+1}} \frac{t_{r+1,u}^{s_{r+1,u}-1}\,dt_{r+1,u}}{e^{t_{r+1,u} + \dots + t_{r+1,n_{r+1}}} -1}. \nonumber
\end{align}
This identity is a specialization of identity (\ref{eq:Uintrep}) below. We can obtain the formula (\ref{eq:Mintrep}) as a consequence of identity (\ref{eq:Uintrep}).
\end{remark}

In this section, we discuss the function (\ref{eq:AKZofM}), and obtain certain relations among Miyagawa-type MZVs. Namely we may regard that the argument in this section a Miyagawa-type analog of Ito's work.

In the two subsections in this section, we show a relationship between special values of the function (\ref{eq:AKZofM}) and  Miyagawa-type MZVs, and a relationship between functions (\ref{eq:AKZofM}) and Miyagawa-type MZFs. In conclusion, we obtain certain relations among Miyagawa-type MZVs. Therefore, we may regard the function (\ref{eq:AKZofM}) as an analog of the Arakawa-Kaneko zeta function of Miyagawa-type.

\subsection{Special values}
The special values of the function (\ref{eq:AKZofM}) can be written in terms of Miyagawa-type MZVs as follows.
\begin{theorem}\label{th:AKZofMspv}
For $\mathbf{k} = (k_{1},\dots,k_{n}) \in \mathbb{N}^{n}, \, k_{n+1} \in \mathbb{N}$ and $m \in \mathbb{Z}_{\ge 0}$, we have
\begin{align*}
&\xi_{MT}((\mathbf{k}; {k}_{n+1}); m+1) \\
&=\sum_{a_{1}+\dots+a_{k_{n+1}+1}=m} \frac{1}{a_{k_{n+1}+1}!}\zeta_{MT}(\{1\}^{a_{k_{n+1}+1}},\mathbf{k}; {}_{-}((a_{1}+1, \dots, a_{k_{n+1}}+1, 2)^{*})),
\end{align*}
where the sum is over all $a_{1},\dots,a_{k_{n+1}+1}\in \mathbb{Z}_{\ge 0}$ satisfying $a_{1}+\dots+a_{k_{n+1}+1}=m$.
\end{theorem}
To prove this theorem, we need the following lemma.
\begin{lemma}\label{le:Mzetaint}
Let $\mathbf{k}_{r+1} = (k_{r+1,1},\dots, k_{r+1,n_{r+1}}) \in \mathbb{N}^{n_{r+1}}${\rm (}$n_{r+1} = 1$ {\rm or}  $k_{r+1,n_{r+1}} \ge 2${\rm )} and $(({}_{+} \mathbf{k}_{r+1})^{*})_{-} = (l_{1}, \dots, l_{d})$. We have
\begin{align*}
&\zeta_{MT}(k_{1},\dots,k_{r};\mathbf{k}_{r+1}) \\
&=\frac{1}{\prod_{i=1}^{d}\Gamma(l_{i})}\int_{0}^{\infty}\cdots\int_{0}^{\infty}\left(\prod_{i=1}^{d}\frac{t_{i}^{l_{i}-1}}{e^{t_{i}+\dots+t_{d}}-1}\right) \prod_{i=1}^{r}{\rm Li}_{k_{i}}(1-e^{-t_{d}})\,dt_{1} \cdots dt_{d}.
\end{align*}
\end{lemma}
\begin{proof}
Let $z \in \mathbb{R}_{\ge 0}$ and $\mathbf{k}_{r+1}=(b_{0},\{1\}^{a_{1}-1},b_{1}+1,\dots,\{1\}^{a_{h}-1},b_{h}+1)$
with $b_{0},a_{i},b_{i} \in \mathbb{N}\ (1 \le i \le h)$.
By using the same method as in the proof of equation $(2.9)$ in Kaneko-Tsumura \cite{KanekoTsumura}, we have
\begin{align}
&{\rm Li}_{k_{1},\dots,k_{r};\mathbf{k}_{r+1}}(1-e^{-z}) \label{eq:liintKTv}\\
&=\int_{0<t_2<\cdots<t_{b_0+\dots+b_h}<z} \nonumber\\
&\quad\left(\prod_{i=1}^{b_{h}}\frac{1}{e^{t_{b_{0}+\dots+b_{h-1}+i}} - 1}\right)
\frac{1}{a_h!}(t_{b_{0}+\dots+b_{h-1}+1}-t_{b_{0}+\dots+b_{h-1}})^{a_h}\nonumber \\
&\quad\times \left(\prod_{i=1}^{b_{h-1}}\frac{1}{e^{t_{b_{0}+\dots+b_{h-2} + i}}-1}\right)\frac{1}{a_{h-1}!}(t_{b_{0}+\dots+b_{h-2}+1} -t_{b_{0}+\dots+b_{h-2}})^{a_{h-1}}\cdots \nonumber \\
&\quad\cdots \times\left(\prod_{i=1}^{b_{1}}\frac{1}{e^{t_{b_{0}+i}}-1}\right)
\frac{1}{a_{1}!}(t_{b_{0}+1}-t_{b_{0}})^{a_{1}} \nonumber \\
&\quad\times \left(\prod_{i=1}^{b_{0}}\frac{1}{e^{t_{i}}-1}\right)\prod_{i=1}^{r} {\rm Li}_{k_{i}}(1-e^{-t_{1}})\,dt_{1} \cdots dt_{b_0+\dots+b_h}. \nonumber
\end{align}
Since 
\[(({}_{+}\mathbf{k}_{n_{r+1}})^{*})_{-}=(\{1\}^{b_{h}-1},a_{h}+1,\dots,\{1\}^{b_{1}-1},a_{1}+1,\{1\}^{b_{0}})=(l_{1},\dots,l_{d}),\] by taking the limit as $z$ tends to infinity, we obtain
\begin{align*}
&\zeta_{MT}(k_{1},\dots,k_{r};k_{r+1,1}, \dots, k_{r+1,n_{r+1}})\\
&=\int_{0<t_2<\cdots<t_{b_0+\dots+b_h}<\infty} \nonumber\\
&\quad\left(\prod_{i=1}^{b_{h}}\frac{1}{e^{t_{b_{0}+\dots+b_{h-1}+i}} - 1}\right)
\frac{1}{a_h!}(t_{b_{0}+\dots+b_{h-1}+1}-t_{b_{0}+\dots+b_{h-1}})^{a_h}\nonumber \\
&\quad\times \left(\prod_{i=1}^{b_{h-1}}\frac{1}{e^{t_{b_{0}+\dots+b_{h-2} + i}}-1}\right)\frac{1}{a_{h-1}!}(t_{b_{0}+\dots+b_{h-2}+1} -t_{b_{0}+\dots+b_{h-2}})^{a_{h-1}}\cdots \nonumber \\
&\quad\dots \times\left(\prod_{i=1}^{b_{1}}\frac{1}{e^{t_{b_{0}+i}}-1}\right)
\frac{1}{a_{1}!}(t_{b_{0}+1}-t_{b_{0}})^{a_{1}} \nonumber \\
&\quad\times \left(\prod_{i=1}^{b_{0}}\frac{1}{e^{t_{i}}-1}\right)\prod_{i=1}^{r} {\rm Li}_{k_{i}}(1-e^{-t_{1}})\,dt_{1} \cdots dt_{b_0+\dots+b_h} \nonumber\\
&=\frac{1}{a_{1}!}\cdots\frac{1}{a_{h}!}\int_{0}^{\infty}\cdots\int_{0}^{\infty} \left(\prod_{i=1}^{b_{h}}\frac{1}{e^{t_{b_{0}+\dots+b_{h-1}+i}+\dots+t_{1}} - 1}\right)
t_{b_{0}+\dots+b_{h-1}+1}^{a_h} \\
&\quad\times  \left(\prod_{i=1}^{b_{h-1}}\frac{1}{e^{t_{b_{0}+\dots+b_{h-2} + i}+\dots+t_{1}}-1}\right)t_{b_{0}+\dots+b_{h-2}+1}^{a_{h-1}} \cdots \\
&\quad\dots \times\left(\prod_{i=1}^{b_{1}}\frac{1}{e^{t_{b_{0}+i}+\dots+t_{1}}-1}\right)
t_{b_{0}+1}^{a_{1}} \\
&\quad\times\left(\prod_{i=1}^{b_{0}}\frac{1}{e^{t_{i}+\dots+t_{1}}-1}\right) \prod_{i=1}^{r}{\rm Li}_{k_{i}}(1-e^{-t_{1}})\,dt_{1} \cdots dt_{t_{b_0+\dots+b_h}} \\
&=\frac{1}{\prod_{i=1}^{d}\Gamma(l_{i})}\int_{0}^{\infty}\cdots\int_{0}^{\infty}\left(\prod_{i=1}^{d}\frac{t_{i}^{l_{d+1-i}-1}}{e^{t_{i}+\dots+t_{1}}-1}\right) \prod_{i=1}^{r}{\rm Li}_{k_{i}}(1-e^{-t_{1}})\,dt_{1} \cdots dt_{d} \\
&=\frac{1}{\prod_{i=1}^{d}\Gamma(l_{i})}\int_{0}^{\infty}\cdots\int_{0}^{\infty}\left(\prod_{i=1}^{d}\frac{t_{i}^{l_{i}-1}}{e^{t_{i}+\dots+t_{d}}-1}\right) \prod_{i=1}^{r}{\rm Li}_{k_{i}}(1-e^{-t_{d}})\,dt_{1} \cdots dt_{d}.
\end{align*}
This completes the proof.
\end{proof}
\begin{proof}[Proof of Theorem \ref{th:AKZofMspv}]
For $\Re (s) > 0$, by using the case $n_{r+1}=1$ of the equation (\ref{eq:liintKTv}), we have
\begin{align*}
&\Gamma(s)\xi_{MT}((\mathbf{k};{k}_{n+1});s) \\
&=\int_{0}^{\infty} \frac{t^{s-1}}{e^{t}-1} \int_{0<t_{1}<\dots<t_{k_{n+1}}<t} \left(\prod_{i=1}^{k_{n+1}}\frac{1}{e^{t_{i}}-1}\right) \prod_{i=1}^{n} {\rm Li}_{k_{i}}(1-e^{-t_{1}}) \,dt_{1} \cdots dt_{k_{n+1}} dt \\
&=\int_{0}^{\infty} \cdots \int_{0}^{\infty} \frac{(t_{k_{n+1}+1} + \dots + t_{1})^{s-1}}{e^{t_{k_{n+1}+1} + \dots +t_{1}} -1} \left(\prod_{i=1}^{k_{n+1}}\frac{1}{e^{t_{i}+\dots+t_{1}}-1}\right)\\
&\quad\times\prod_{i=1}^{n} {\rm Li}_{k_{i}}(1-e^{-t_{1}}) \,dt_{1} \cdots dt_{k_{n+1}} dt_{k_{n+1}+1} \\
&=\int_{0}^{\infty} \cdots \int_{0}^{\infty} \frac{(t_{1} + \dots + t_{k_{n+1}+1})^{s-1}}{e^{t_{1} + \dots +t_{k_{n+1}+1}} -1} \left(\prod_{i=2}^{k_{n+1}+1}\frac{1}{e^{t_{i}+\dots+t_{k_{n+1}+1}}-1}\right)\\
&\quad\times\prod_{i=1}^{n} {\rm Li}_{k_{i}}(1-e^{-t_{k_{n+1}+1}}) \,dt_{1} \cdots dt_{k_{n+1}} dt_{k_{n+1}+1}.
\end{align*}
By putting $s=m+1$ in this equation and using the formula ${\rm Li}_{1}(1-e^{-t}) = t$ and  Lemma \ref{le:Mzetaint}, we have
\begin{align*}
&m!\xi_{MT}((\mathbf{k};{k}_{n+1});m+1) \\
&=\int_{0}^{\infty} \cdots \int_{0}^{\infty} (t_{1} + \dots + t_{k_{n+1}+1})^{m} \left(\prod_{i=1}^{k_{n+1}+1}\frac{1}{e^{t_{i}+\dots+t_{k_{n+1}+1}}-1}\right)\\
&\quad\times\prod_{i=1}^{n} {\rm Li}_{k_{i}}(1-e^{-t_{k_{n+1}+1}}) \,dt_{1} \cdots dt_{k_{n+1}} dt_{k_{n+1}+1}\\
&=\sum_{a_{1}+\dots+a_{k_{n+1}+1}=m} \frac{m!}{a_{1}! \cdots a_{k_{n+1}+1}!} \\
&\quad\times\int_{0}^{\infty} \cdots  \int_{0}^{\infty} t_{1}^{a_{1}}  \cdots  t_{k_{n+1}}^{a_{k_{n+1}}} {\Bigl( {\rm Li}_{1}(1  -  e^{-t_{k_{n+1}+1}}) \Bigr)}^{a_{k_{n+1}+1}} \\
&\quad\times\left(\prod_{i=1}^{k_{n+1}+1}\frac{1}{e^{t_{i}+\dots+t_{k_{n+1}+1}}-1}\right)\prod_{i=1}^{n} {\rm Li}_{k_{i}}(1-e^{-t_{k_{n+1}+1}}) \,dt_{1} \cdots dt_{k_{n+1}} dt_{k_{n+1}+1}\\
&=\sum_{a_{1}+\dots+a_{k_{n+1}+1}=m} \frac{m!}{a_{k_{n+1}+1}!} \zeta(\{1\}^{a_{k_{n+1}+1}} , \mathbf{k};{}_{-}((a_{1}+1,\dots,a_{k_{n+1}}+1,2)^{*})).
\end{align*}
This completes the proof.
\end{proof}
\begin{remark}
We can also prove Theorem \ref{th:AKZofMspv} by using the Yamamoto-integral defined by Yamamoto \cite{Yamamoto}. This method is more intuitive. Here, we use the notation in \cite{Yamamoto}. Since
\begin{align*}
m!\xi_{MT}((\mathbf{k}; {k}_{n+1}); m+1) &=\int_{0}^{\infty} \frac{\bigl({\rm Li}_{1}(1-e^{-t})\bigr)^{m}}{e^{t}-1} {\rm Li}_{\mathbf{k};k_{n+1}} (1-e^{-t}) \,dt \\
&= \int_{0}^{1} \frac{\bigl({\rm Li}_{1} (t) \bigr)^{m}}{t} {\rm Li}_{\mathbf{k};k_{n+1}} (t) \,dt,
\end{align*}
by using the Yamamoto-integral, we have
\begin{align*}
&m!\xi_{MT}((\mathbf{k}; {k}_{n+1}); m+1) \\
&= I\left(
\begin{tikzpicture}[thick, baseline=-18.5mm]
{
\Whitecircle (E0) at (0,0) {};
\Blackcircle (E0_0_0_0) at ($(E0) + ( - 1 * \BRANCH_WIDTH*0.5, -\LINE_DISTANCE)$) {};
\Blackcircle (E0_0_0_1) at ($(E0) + ( 1 * \BRANCH_WIDTH*0.5, -\LINE_DISTANCE)$) {};
\draw [line width = \LINEWIDTH] (E0) -- (E0_0_0_0) (E0) -- (E0_0_0_1);
\tdot (d0_1) at ($(E0_0_0_0) + (\DOTSBRANK*0.5, \DOTSPOSITION)$) {} ;
\tdot (d0_3) at ($(E0_0_0_1) + (- \DOTSBRANK*0.5, \DOTSPOSITION)$){};
\tdot (d0_2) at ($(d0_1)!.5!(d0_3)$){};
\Whitecircle (E0_0_0) at ($(E0) + ( - 1 * \BRANCH_WIDTH, -\LINE_DISTANCE)$) {};
\draw [line width = \LINEWIDTH] (E0) -- (E0_0_0);
\Whitecircle (E0_0_1) at ($(E0_0_0) + (0, -\DOT_DISTANCE)$) {};
\drawvdots{E0_0_0}{E0_0_1};
\def\branch{E0_0_1};
\Whitecircle (E1_0_0) at ($(\branch) + ( - 1 * \BRANCH_WIDTH, -\LINE_DISTANCE)$) {};
\draw  [line width = \LINEWIDTH](\branch) -- (E1_0_0);
\Whitecircle (E1_0_1) at ($(E1_0_0) + (0, - \DOT_DISTANCE)$) {};
\drawvdots{E1_0_0}{E1_0_1}
\Blackcircle (E1_0_2) at ($(E1_0_1) + (0, - \LINE_DISTANCE)$) {};
\draw [line width = \LINEWIDTH] (E1_0_1) -- (E1_0_2);
\Whitecircle (E3_0_0) at ($(\branch) + (2 * \BRANCH_WIDTH - 1 * \BRANCH_WIDTH, -\LINE_DISTANCE)$) {};
\draw [line width = \LINEWIDTH] (\branch) -- (E3_0_0);
\Whitecircle (E3_0_1) at ($(E3_0_0) + (0, - \DOT_DISTANCE)$) {};
\drawvdots{E3_0_0}{E3_0_1}
\Blackcircle (E3_0_2) at ($(E3_0_1) + (0, - \LINE_DISTANCE)$) {};
\draw [line width = \LINEWIDTH] (E3_0_1) -- (E3_0_2);
\Whitecircle (E2_0_0) at ($(E1_0_0)!\SECOND_EDGE_PROPOTION!(E3_0_0)$) {};
\draw  [line width = \LINEWIDTH] (\branch) -- (E2_0_0);
\Whitecircle (E2_0_1) at ($(E2_0_0) + (0, - \DOT_DISTANCE)$) {};
\drawvdots{E2_0_0}{E2_0_1}
\Blackcircle (E2_0_2) at ($(E2_0_1) + (0, - \LINE_DISTANCE)$) {};
\draw [line width = \LINEWIDTH] (E2_0_1) -- (E2_0_2);
\draw[decorate, decoration={brace, mirror, amplitude = \BRACEAMPLITUDE}]
($(E0_0_0_0) + (- \BRACEVSPACE, - \BRACEHSPACE)$) -- ($(E0_0_0_1) + ( \BRACEVSPACE,  -\BRACEHSPACE) $) node[xshift = 0mm, yshift = -\BRACEXSIFT*0.6, pos = 0.5] {\INDEX_SIZE$m$};
\draw[decorate, decoration={brace, mirror, amplitude = \BRACEAMPLITUDE}]
($(E0_0_0) + (- \BRACEVSPACE, \BRACEHSPACE)$) -- ($(E0_0_1) + (- \BRACEVSPACE, - \BRACEHSPACE) $) node[xshift = -\BRACEXSIFT, yshift = 0mm, pos = 0.5] {\INDEX_SIZE$k_{n+1}$};
\draw[decorate, decoration={brace, mirror, amplitude = \BRACEAMPLITUDE}]
($(E1_0_0) + ( - \BRACEVSPACE, \BRACEHSPACE)$) -- ($(E1_0_1) + ( - \BRACEVSPACE, - \BRACEHSPACE) $) node[xshift = - \BRACEXSIFT, yshift = 0mm, pos = 0.5] {\INDEX_SIZE$k_{1} - 1$};
\draw[decorate, decoration={brace, amplitude = \BRACEAMPLITUDE}]
($(E2_0_0) + (\BRACEVSPACE, \BRACEHSPACE)$) -- ($(E2_0_1) + (\BRACEVSPACE, - \BRACEHSPACE) $) node[xshift = \BRACEXSIFT, yshift = 0mm, pos = 0.5] {\INDEX_SIZE$k_{2} - 1$};
\draw[decorate, decoration={brace, amplitude = \BRACEAMPLITUDE}]
($(E3_0_0) + (\BRACEVSPACE, \BRACEHSPACE)$) -- ($(E3_0_1) + (\BRACEVSPACE, - \BRACEHSPACE) $) node[xshift = \BRACEXSIFT, yshift = 0mm, pos = 0.5] {\INDEX_SIZE$k_{n} - 1$};
\tdot (d1) at ($(E2_0_1)!.5!(E2_0_2) + (\DOTSBRANK, \DOTSPOSITION)$) {} ;
\tdot (d3) at ($(E3_0_1)!.5!(E3_0_2) + (- \DOTSBRANK, \DOTSPOSITION)$){};
\tdot (d2) at ($(d1)!.5!(d3)$){};
}\end{tikzpicture}
\right) \\
&= \sum_{a_{1}+\dots+a_{k_{n+1}+1}=m}\frac{m!}{a_{k_{n+1}+1}!}I\left(
\begin{tikzpicture}[thick, baseline=-35.5mm]
{
\Whitecircle (E0_0_0) at (0,0) {};
\Blackcircle (E0_0_1) at ($(E0_0_0) + (0, -\LINE_DISTANCE)$) {};
\draw  [line width = \LINEWIDTH](E0_0_0) -- (E0_0_1);
\Blackcircle (E0_0_2) at ($(E0_0_1) + (0, -\DOT_DISTANCE)$) {};
\drawvdots{E0_0_1}{E0_0_2};
\Whitecircle (E0_0_3) at ($(E0_0_2) + (0, -\LINE_DISTANCE)$) {};
\draw  [line width = \LINEWIDTH](E0_0_2) -- (E0_0_3);
\Whitecircle (E0_0_4) at ($(E0_0_3) + (0, -\DOT_DISTANCE)$) {};
\drawvdots{E0_0_3}{E0_0_4};
\Blackcircle (E0_0_5) at ($(E0_0_4) + (0, -\LINE_DISTANCE)$) {};
\draw  [line width = \LINEWIDTH](E0_0_4) -- (E0_0_5);
\Blackcircle (E0_0_6) at ($(E0_0_5) + (0, -\DOT_DISTANCE)$) {};
\drawvdots{E0_0_5}{E0_0_6};
\Whitecircle (E0_0_7) at ($(E0_0_6) + (0, -\LINE_DISTANCE)$) {};
\draw  [line width = \LINEWIDTH](E0_0_6) -- (E0_0_7);
\def\branch{E0_0_7};
\Whitecircle (E1_0_0) at ($(\branch) + ( -1.5 *  \BRANCH_WIDTH, -\LINE_DISTANCE)$) {};
\draw  [line width = \LINEWIDTH](\branch) -- (E1_0_0);
\Whitecircle (E1_0_1) at ($(E1_0_0) + (0, - \DOT_DISTANCE)$) {};
\drawvdots{E1_0_0}{E1_0_1}
\Blackcircle (E1_0_2) at ($(E1_0_1) + (0, - \LINE_DISTANCE)$) {};
\draw [line width = \LINEWIDTH] (E1_0_1) -- (E1_0_2);
\Whitecircle (E3_0_0) at ($(\branch) + (0.5 * \BRANCH_WIDTH, -\LINE_DISTANCE)$) {};
\draw [line width = \LINEWIDTH] (\branch) -- (E3_0_0);
\Whitecircle (E3_0_1) at ($(E3_0_0) + (0, - \DOT_DISTANCE)$) {};
\drawvdots{E3_0_0}{E3_0_1}
\Blackcircle (E3_0_2) at ($(E3_0_1) + (0, - \LINE_DISTANCE)$) {};
\draw [line width = \LINEWIDTH] (E3_0_1) -- (E3_0_2);
\Whitecircle (E2_0_0) at ($(E1_0_0)!\SECOND_EDGE_PROPOTION!(E3_0_0)$) {};
\draw  [line width = \LINEWIDTH] (\branch) -- (E2_0_0);
\Whitecircle (E2_0_1) at ($(E2_0_0) + (0, - \DOT_DISTANCE)$) {};
\drawvdots{E2_0_0}{E2_0_1}
\Blackcircle (E2_0_2) at ($(E2_0_1) + (0, - \LINE_DISTANCE)$) {};
\draw [line width = \LINEWIDTH] (E2_0_1) -- (E2_0_2);
\Blackcircle (E4_0_0) at ($(E3_0_0) + (\BRANCH_WIDTH, 0)$) {};
\draw  [line width = \LINEWIDTH] (\branch) -- (E4_0_0);
\Blackcircle (E5_0_0) at ($(E4_0_0) + (\BRANCH_WIDTH, 0)$) {};
\draw  [line width = \LINEWIDTH] (\branch) -- (E5_0_0);
\tdot (d0_1) at ($(E4_0_0) + (\DOTSBRANK*0.5, \DOTSPOSITION-0.05 )$) {} ;
\tdot (d0_3) at ($(E5_0_0) + (- \DOTSBRANK*0.5, \DOTSPOSITION-0.05)$){};
\tdot (d0_2) at ($(d0_1)!.5!(d0_3)$){};
\draw[decorate, decoration={brace, mirror, amplitude = \BRACEAMPLITUDE}]
($(E0_0_1) + (- \BRACEVSPACE, \BRACEHSPACE)$) -- ($(E0_0_2) + (- \BRACEVSPACE, - \BRACEHSPACE) $) node[xshift = -\BRACEXSIFT, yshift = 0mm, pos = 0.5] {\INDEX_SIZE$a_{1}$};
\draw[decorate, decoration={brace, mirror, amplitude = \BRACEAMPLITUDE}]
($(E0_0_5) + (- \BRACEVSPACE, \BRACEHSPACE)$) -- ($(E0_0_6) + (- \BRACEVSPACE, - \BRACEHSPACE) $) node[xshift = -\BRACEXSIFT, yshift = 0mm, pos = 0.5] {\INDEX_SIZE$a_{k_{n+1}}$};
\draw[decorate, decoration={brace, mirror, amplitude = \BRACEAMPLITUDE}]
($(E1_0_0) + ( - \BRACEVSPACE, \BRACEHSPACE)$) -- ($(E1_0_1) + ( - \BRACEVSPACE, - \BRACEHSPACE) $) node[xshift = - \BRACEXSIFT, yshift = 0mm, pos = 0.5] {\INDEX_SIZE$k_{1} - 1$};
\draw[decorate, decoration={brace, amplitude = \BRACEAMPLITUDE}]
($(E2_0_0) + (\BRACEVSPACE, \BRACEHSPACE)$) -- ($(E2_0_1) + (\BRACEVSPACE, - \BRACEHSPACE) $) node[xshift = \BRACEXSIFT, yshift = 0mm, pos = 0.5] {\INDEX_SIZE$k_{2} - 1$};
\draw[decorate, decoration={brace, amplitude = \BRACEAMPLITUDE}]
($(E3_0_0) + (\BRACEVSPACE, \BRACEHSPACE)$) -- ($(E3_0_1) + (\BRACEVSPACE, - \BRACEHSPACE) $) node[xshift = \BRACEXSIFT, yshift = 0mm, pos = 0.5] {\INDEX_SIZE$k_{n} - 1$};
\draw[decorate, decoration={brace, mirror, amplitude = \BRACEAMPLITUDE}]
($(E4_0_0) + (- \BRACEVSPACE, - \BRACEHSPACE)$) -- ($(E5_0_0) + ( \BRACEVSPACE,  -\BRACEHSPACE) $) node[xshift = 5mm, yshift = -\BRACEXSIFT*0.6, pos = 0.5] {\INDEX_SIZE$a_{k_{n+1}+1}$};
\tdot (d1) at ($(E2_0_1)!.5!(E2_0_2) + (\DOTSBRANK, \DOTSPOSITION)$) {} ;
\tdot (d3) at ($(E3_0_1)!.5!(E3_0_2) + (- \DOTSBRANK, \DOTSPOSITION)$){};
\tdot (d2) at ($(d1)!.5!(d3)$){};
}\end{tikzpicture}
\right).
\end{align*}
The second equality is obtained by ordering $m$ black vertices and $1+k_{r+1}$ white vertices. 
Here, by Lemma \ref{le:Lidi}, the special values of the Miyagawa-type MZF is written as 
\begin{align*}
&\zeta_{MT}(k_{1}, \dots, k_{r}; \mathbf{k}_{r+1}) \\
&=\int_{0<t_{1}<\dots<t_{k_{r+1,1} +\dots+k_{r+1,n_{r+1}}}<1} \left(\prod_{i=2}^{k_{r+1},n_{r+1}}\frac{1}{t_{k_{r+1,1} +\dots+k_{r+1,n_{r+1}-1} +i}}\right) \\
&\quad\times\frac{1}{1-t_{k_{r+1,1} +\dots+k_{r+1,n_{r+1}-1} +1}} \left(\prod_{i=2}^{k_{r+1,n_{r+1}-1}}\frac{1}{t_{k_{r+1,1} +\dots+k_{r+1,n_{r+1}-2} +i}}\right)\cdots\\
&\quad\dots \times \frac{1}{1-t_{k_{r+1,1}+k_{r+1,2}+1}} \left(\prod_{i=2}^{k_{r+1,2}}\frac{1}{t_{k_{r+1,1}+i}}\right)\\
&\quad\times\frac{1}{1-t_{k_{r+1,1}+1}} \left(\prod_{i=1}^{k_{r+1,1}}\frac{1}{t_{i}}\right) \prod_{i=1}^{r}{\rm Li}_{k_i}(t_{1})\,dt_{1}\cdots dt_{k_{r+1,1}+\dots+k_{r+1,n_{r+1}}}\\
&= I\left(
\begin{tikzpicture}[thick, baseline=-40mm]
{
\Whitecircle (E0_0_0) at (0,0) {};
\Whitecircle (E0_0_1) at ($(E0_0_0) + (0, -\DOT_DISTANCE)$) {};
\drawvdots{E0_0_0}{E0_0_1};
\Blackcircle (E0_0_2) at ($(E0_0_1) + (0, -\LINE_DISTANCE)$) {};
\draw [line width = \LINEWIDTH] (E0_0_1) -- (E0_0_2);
\Blackcircle (E0_0_3) at ($(E0_0_2) + (0, -\DOT_DISTANCE)$) {};
\drawvdots{E0_0_2}{E0_0_3};
\Whitecircle (E0_0_4) at ($(E0_0_3) + (0, -\LINE_DISTANCE)$) {};
\draw [line width = \LINEWIDTH] (E0_0_3) -- (E0_0_4);
\Whitecircle (E0_0_5) at ($(E0_0_4) + (0, -\DOT_DISTANCE)$) {};
\drawvdots{E0_0_4}{E0_0_5};
\Blackcircle (E0_0_6) at ($(E0_0_5) + (0, -\LINE_DISTANCE)$) {};
\draw [line width = \LINEWIDTH](E0_0_5)  -- (E0_0_6);
\Whitecircle (E0_0_7) at ($(E0_0_6) + (0, -\LINE_DISTANCE)$) {};
\draw [line width = \LINEWIDTH] (E0_0_6) -- (E0_0_7);
\Whitecircle (E0_0_8) at ($(E0_0_7) + (0, -\DOT_DISTANCE)$) {};
\drawvdots{E0_0_7}{E0_0_8};
\def\branch{E0_0_8};
\Whitecircle (E1_0_0) at ($(\branch) + ( -1 *  \BRANCH_WIDTH, -\LINE_DISTANCE)$) {};
\draw  [line width = \LINEWIDTH](\branch) -- (E1_0_0);
\Whitecircle (E1_0_1) at ($(E1_0_0) + (0, - \DOT_DISTANCE)$) {};
\drawvdots{E1_0_0}{E1_0_1}
\Blackcircle (E1_0_2) at ($(E1_0_1) + (0, - \LINE_DISTANCE)$) {};
\draw [line width = \LINEWIDTH] (E1_0_1) -- (E1_0_2);
\Whitecircle (E3_0_0) at ($(\branch) + ( \BRANCH_WIDTH, -\LINE_DISTANCE)$) {};
\draw [line width = \LINEWIDTH] (\branch) -- (E3_0_0);
\Whitecircle (E3_0_1) at ($(E3_0_0) + (0, - \DOT_DISTANCE)$) {};
\drawvdots{E3_0_0}{E3_0_1}
\Blackcircle (E3_0_2) at ($(E3_0_1) + (0, - \LINE_DISTANCE)$) {};
\draw [line width = \LINEWIDTH] (E3_0_1) -- (E3_0_2);
\Whitecircle (E2_0_0) at ($(E1_0_0)!\SECOND_EDGE_PROPOTION!(E3_0_0)$) {};
\draw  [line width = \LINEWIDTH] (\branch) -- (E2_0_0);
\Whitecircle (E2_0_1) at ($(E2_0_0) + (0, - \DOT_DISTANCE)$) {};
\drawvdots{E2_0_0}{E2_0_1}
\Blackcircle (E2_0_2) at ($(E2_0_1) + (0, - \LINE_DISTANCE)$) {};
\draw [line width = \LINEWIDTH] (E2_0_1) -- (E2_0_2);
\draw[decorate, decoration={brace, mirror, amplitude = \BRACEAMPLITUDE}]
($(E0_0_0) + (- \BRACEVSPACE, \BRACEHSPACE)$) -- ($(E0_0_1) + (- \BRACEVSPACE, - \BRACEHSPACE) $) node[xshift = -\BRACEXSIFT*1.7, yshift = 0mm, pos = 0.5] {\INDEX_SIZE$k_{r+1,n_{r+1}}-1$};
\draw[decorate, decoration={brace, mirror, amplitude = \BRACEAMPLITUDE}]
($(E0_0_4) + (- \BRACEVSPACE, \BRACEHSPACE)$) -- ($(E0_0_5) + (- \BRACEVSPACE, - \BRACEHSPACE) $) node[xshift = -\BRACEXSIFT*1.3, yshift = 0mm, pos = 0.5] {\INDEX_SIZE$k_{r+1,2}-1$};
\draw[decorate, decoration={brace, mirror, amplitude = \BRACEAMPLITUDE}]
($(E0_0_7) + (- \BRACEVSPACE, \BRACEHSPACE)$) -- ($(E0_0_8) + (- \BRACEVSPACE, - \BRACEHSPACE) $) node[xshift = -\BRACEXSIFT*1.3, yshift = 0mm, pos = 0.5] {\INDEX_SIZE$k_{r+1,1}$};
\draw[decorate, decoration={brace, mirror, amplitude = \BRACEAMPLITUDE}]
($(E1_0_0) + ( - \BRACEVSPACE, \BRACEHSPACE)$) -- ($(E1_0_1) + ( - \BRACEVSPACE, - \BRACEHSPACE) $) node[xshift = - \BRACEXSIFT, yshift = 0mm, pos = 0.5] {\INDEX_SIZE$k_{1} - 1$};
\draw[decorate, decoration={brace, amplitude = \BRACEAMPLITUDE}]
($(E2_0_0) + (\BRACEVSPACE, \BRACEHSPACE)$) -- ($(E2_0_1) + (\BRACEVSPACE, - \BRACEHSPACE) $) node[xshift = \BRACEXSIFT, yshift = 0mm, pos = 0.5] {\INDEX_SIZE$k_{2} - 1$};
\draw[decorate, decoration={brace, amplitude = \BRACEAMPLITUDE}]
($(E3_0_0) + (\BRACEVSPACE, \BRACEHSPACE)$) -- ($(E3_0_1) + (\BRACEVSPACE, - \BRACEHSPACE) $) node[xshift = \BRACEXSIFT, yshift = 0mm, pos = 0.5] {\INDEX_SIZE$k_{r} - 1$};
\tdot (d1) at ($(E2_0_1)!.5!(E2_0_2) + (\DOTSBRANK, \DOTSPOSITION)$) {} ;
\tdot (d3) at ($(E3_0_1)!.5!(E3_0_2) + (- \DOTSBRANK, \DOTSPOSITION)$){};
\tdot (d2) at ($(d1)!.5!(d3)$){};
}\end{tikzpicture}
\right).
\end{align*}
Therefore, we obtain
\begin{align*}
&\zeta_{MT}(\{1\}^{a_{k_{n+1}+1}}, k_{1}, \dots, k_{n}; {}_{-}((a_{1}+1, \dots, a_{k_{n+1}}+1, 2)^{*})) \\
&= I\left(
\begin{tikzpicture}[thick, baseline=-35.5mm]
{
\Whitecircle (E0_0_0) at (0,0) {};
\Blackcircle (E0_0_1) at ($(E0_0_0) + (0, -\LINE_DISTANCE)$) {};
\draw  [line width = \LINEWIDTH](E0_0_0) -- (E0_0_1);
\Blackcircle (E0_0_2) at ($(E0_0_1) + (0, -\DOT_DISTANCE)$) {};
\drawvdots{E0_0_1}{E0_0_2};
\Whitecircle (E0_0_3) at ($(E0_0_2) + (0, -\LINE_DISTANCE)$) {};
\draw  [line width = \LINEWIDTH](E0_0_2) -- (E0_0_3);
\Whitecircle (E0_0_4) at ($(E0_0_3) + (0, -\DOT_DISTANCE)$) {};
\drawvdots{E0_0_3}{E0_0_4};
\Blackcircle (E0_0_5) at ($(E0_0_4) + (0, -\LINE_DISTANCE)$) {};
\draw  [line width = \LINEWIDTH](E0_0_4) -- (E0_0_5);
\Blackcircle (E0_0_6) at ($(E0_0_5) + (0, -\DOT_DISTANCE)$) {};
\drawvdots{E0_0_5}{E0_0_6};
\Whitecircle (E0_0_7) at ($(E0_0_6) + (0, -\LINE_DISTANCE)$) {};
\draw  [line width = \LINEWIDTH](E0_0_6) -- (E0_0_7);
\def\branch{E0_0_7};
\Whitecircle (E1_0_0) at ($(\branch) + ( -1.5 *  \BRANCH_WIDTH, -\LINE_DISTANCE)$) {};
\draw  [line width = \LINEWIDTH](\branch) -- (E1_0_0);
\Whitecircle (E1_0_1) at ($(E1_0_0) + (0, - \DOT_DISTANCE)$) {};
\drawvdots{E1_0_0}{E1_0_1}
\Blackcircle (E1_0_2) at ($(E1_0_1) + (0, - \LINE_DISTANCE)$) {};
\draw [line width = \LINEWIDTH] (E1_0_1) -- (E1_0_2);
\Whitecircle (E3_0_0) at ($(\branch) + (0.5 * \BRANCH_WIDTH, -\LINE_DISTANCE)$) {};
\draw [line width = \LINEWIDTH] (\branch) -- (E3_0_0);
\Whitecircle (E3_0_1) at ($(E3_0_0) + (0, - \DOT_DISTANCE)$) {};
\drawvdots{E3_0_0}{E3_0_1}
\Blackcircle (E3_0_2) at ($(E3_0_1) + (0, - \LINE_DISTANCE)$) {};
\draw [line width = \LINEWIDTH] (E3_0_1) -- (E3_0_2);
\Whitecircle (E2_0_0) at ($(E1_0_0)!\SECOND_EDGE_PROPOTION!(E3_0_0)$) {};
\draw  [line width = \LINEWIDTH] (\branch) -- (E2_0_0);
\Whitecircle (E2_0_1) at ($(E2_0_0) + (0, - \DOT_DISTANCE)$) {};
\drawvdots{E2_0_0}{E2_0_1}
\Blackcircle (E2_0_2) at ($(E2_0_1) + (0, - \LINE_DISTANCE)$) {};
\draw [line width = \LINEWIDTH] (E2_0_1) -- (E2_0_2);
\Blackcircle (E4_0_0) at ($(E3_0_0) + (\BRANCH_WIDTH, 0)$) {};
\draw  [line width = \LINEWIDTH] (\branch) -- (E4_0_0);
\Blackcircle (E5_0_0) at ($(E4_0_0) + (\BRANCH_WIDTH, 0)$) {};
\draw  [line width = \LINEWIDTH] (\branch) -- (E5_0_0);
\tdot (d0_1) at ($(E4_0_0) + (\DOTSBRANK*0.5, \DOTSPOSITION-0.05 )$) {} ;
\tdot (d0_3) at ($(E5_0_0) + (- \DOTSBRANK*0.5, \DOTSPOSITION-0.05)$){};
\tdot (d0_2) at ($(d0_1)!.5!(d0_3)$){};
\draw[decorate, decoration={brace, mirror, amplitude = \BRACEAMPLITUDE}]
($(E0_0_1) + (- \BRACEVSPACE, \BRACEHSPACE)$) -- ($(E0_0_2) + (- \BRACEVSPACE, - \BRACEHSPACE) $) node[xshift = -\BRACEXSIFT, yshift = 0mm, pos = 0.5] {\INDEX_SIZE$a_{1}$};
\draw[decorate, decoration={brace, mirror, amplitude = \BRACEAMPLITUDE}]
($(E0_0_5) + (- \BRACEVSPACE, \BRACEHSPACE)$) -- ($(E0_0_6) + (- \BRACEVSPACE, - \BRACEHSPACE) $) node[xshift = -\BRACEXSIFT, yshift = 0mm, pos = 0.5] {\INDEX_SIZE$a_{k_{n+1}}$};
\draw[decorate, decoration={brace, mirror, amplitude = \BRACEAMPLITUDE}]
($(E1_0_0) + ( - \BRACEVSPACE, \BRACEHSPACE)$) -- ($(E1_0_1) + ( - \BRACEVSPACE, - \BRACEHSPACE) $) node[xshift = - \BRACEXSIFT, yshift = 0mm, pos = 0.5] {\INDEX_SIZE$k_{1} - 1$};
\draw[decorate, decoration={brace, amplitude = \BRACEAMPLITUDE}]
($(E2_0_0) + (\BRACEVSPACE, \BRACEHSPACE)$) -- ($(E2_0_1) + (\BRACEVSPACE, - \BRACEHSPACE) $) node[xshift = \BRACEXSIFT, yshift = 0mm, pos = 0.5] {\INDEX_SIZE$k_{2} - 1$};
\draw[decorate, decoration={brace, amplitude = \BRACEAMPLITUDE}]
($(E3_0_0) + (\BRACEVSPACE, \BRACEHSPACE)$) -- ($(E3_0_1) + (\BRACEVSPACE, - \BRACEHSPACE) $) node[xshift = \BRACEXSIFT, yshift = 0mm, pos = 0.5] {\INDEX_SIZE$k_{n} - 1$};
\draw[decorate, decoration={brace, mirror, amplitude = \BRACEAMPLITUDE}]
($(E4_0_0) + (- \BRACEVSPACE, - \BRACEHSPACE)$) -- ($(E5_0_0) + ( \BRACEVSPACE,  -\BRACEHSPACE) $) node[xshift = 5mm, yshift = -\BRACEXSIFT*0.6, pos = 0.5] {\INDEX_SIZE$a_{k_{n+1}+1}$};
\tdot (d1) at ($(E2_0_1)!.5!(E2_0_2) + (\DOTSBRANK, \DOTSPOSITION)$) {} ;
\tdot (d3) at ($(E3_0_1)!.5!(E3_0_2) + (- \DOTSBRANK, \DOTSPOSITION)$){};
\tdot (d2) at ($(d1)!.5!(d3)$){};
}\end{tikzpicture}
\right).
\end{align*}
Therefore, we obtain Theorem \ref{th:AKZofMspv}.
\end{remark}
\subsection{Functional relations}
Functions (\ref{eq:AKZofM}) has the relationship with Miyagawa-type MZFs as follows.
\begin{theorem}\label{th:AKZofMfre}
For $l,\, k \in \mathbb{N},\ s \in \mathbb{C}$, we have
\begin{align*}
&\zeta(2)^{l} \zeta(\{1\}^{k},s) +\sum_{j=1}^{l} \binom{l}{j} \zeta(2)^{l-j} (-1)^{j} \biggl(\sum_{i=1}^{k}(-1)^{i-1}\zeta_{MT}(\{2\}^{j};i)\zeta(\{1\}^{k-i},s) \\
&\quad+(-1)^{k} \xi_{MT}((\{2\}^{j};k);s) \biggr) \\
&=\hspace{-1.3mm}\sum_{a+b_{1}+\dots+b_{k+1}=l}\frac{l!}{a!} \binom{s+b_{k+1}-1}{b_{k+1}} \zeta_{MT}(0,\!\{2\}^{a},\!\{1\}^{l-a};b_{1}+1,\dots,b_{k}+1,b_{k+1}+s),
\end{align*}
where the sum on the right hand side is over all $a \in \mathbb{Z}_{\ge 0}$ and $b_{i} \in \mathbb{Z}_{\ge 0}$ satisfying $a+b_{1}+\dots+b_{k+1}=l$.
\end{theorem}
\begin{remark}
If we understand the sum in $i$ as in equal to $0$ when $k=0$, then Theorem \ref{th:AKZofMfre} holds also when $k=0$ and coincides with Theorem \ref{th:IZfrel}.
\end{remark}
\begin{remark}
By putting $s=m+1$ in Theorem \ref{th:AKZofMfre} and using Theorem \ref{th:AKZofMspv} for $\xi_{MT}((\{2\}^{j};k);s)$, we can obtain relations among Miyagawa-type MZVs.
\end{remark}
\begin{proof}[Proof of Theorem \ref{th:AKZofMfre}]
For $s \in \mathbb{C}$ with $\Re(s) > 1$, let 
\begin{align*}
J&=\int_{0}^{\infty} \cdots \int_{0}^{\infty} t_{k+1}^{s-1} \left(\prod_{i=1}^{l}\frac{u_{i}+t_{1}+\dots+t_{k+1}}{e^{u_{i}+t_{1}+\dots+t_{k+1}}-1}\right) \\
&\quad\times \frac{1}{e^{t_{1}+\dots+t_{k+1}}-1} \frac{1}{e^{t_{2}+\dots+t_{k+1}}-1} \cdots \frac{1}{e^{t_{k+1}}-1} \,du_{1}\cdots du_{l}dt_{1}\cdots dt_{k+1}.
\end{align*}
We calculate $J$ in two different ways.

The first calculation is to integrate directly by using Corollary \ref{co:liint}. By integrating with respect to $u_{1},\dots,u_{m}$, we have
\begin{align*}
J&=\int_{0}^{\infty} \cdots \int_{0}^{\infty} t_{k+1}^{s-1} \Bigl(\zeta(2) - {\rm Li}_{2}(1-e^{-(t_{1}+\dots+t_{k+1})})\Bigr)^{l} \\
&\quad\times \frac{1}{e^{t_{1}+\dots+t_{k+1}}-1} \frac{1}{e^{t_{2}+\dots+t_{k+1}}-1} \cdots \frac{1}{e^{t_{k+1}}-1} \,dt_{1}\cdots dt_{k+1} \\
&=\Gamma(s)\zeta(2)^{l} \zeta(\{1\}^{k},s) \\
&\quad+\sum_{j=1}^{l} \binom{l}{j} \zeta(2)^{l-j} (-1)^{j}\int_{0}^{\infty} \cdots \int_{0}^{\infty} t_{k+1}^{s-1} \Bigl( {\rm Li}_{2}(1-e^{-(t_{1}+\dots+t_{k+1})}) \Bigr)^{j} \\
&\quad\times \frac{1}{e^{t_{1}+\dots+t_{k+1}}-1} \frac{1}{e^{t_{2}+\dots+t_{k+1}}-1} \cdots \frac{1}{e^{t_{k+1}}-1} \,dt_{1}\cdots dt_{k+1}.
\end{align*}
We integrate the above integral by parts in order of $t_{1},\dots,t_{k}$ to obtain
\begin{align*}
&\int_{0}^{\infty}  \cdots  \int_{0}^{\infty} t_{k+1}^{s-1} \Bigl( {\rm Li}_{2}(1-e^{-(t_{1}+\dots+t_{k+1})}) \Bigr)^{j} \\
&\quad\times \frac{1}{e^{t_{1}+\dots+t_{k+1}}-1}\frac{1}{e^{t_{2}+\dots+t_{k+1}}-1}  \cdots   \frac{1}{e^{t_{k+1}}-1} \,dt_{1}\cdots dt_{k+1} \\
&=\int_{0}^{\infty} \cdots \int_{0}^{\infty} t_{k+1}^{s-1}  \Bigl( \zeta_{MT}(\{2\}^{j};1) - {\rm Li}_{\{2\}^{j};1} (1-e^{-(t_{2}+\dots+t_{k+1})}) \Bigr) \\
&\quad\times \frac{1}{e^{t_{2}+\dots+t_{k+1}}-1}\frac{1}{e^{t_{3}+\dots+t_{k+1}}-1} \cdots \frac{1}{e^{t_{k+1}}-1} \,dt_{2}\cdots dt_{k+1} \\
&=\Gamma(s)\zeta_{MT}(\{2\}^{j};1)\zeta(\{1\}^{k-1},s) \\
&\quad-\int_{0}^{\infty}  \cdots  \int_{0}^{\infty} t_{k+1}^{s-1}   {\rm Li}_{\{2\}^{j};1} (1-e^{-(t_{2}+\dots+t_{k+1})}) \\
&\quad\times \frac{1}{e^{t_{2}+\dots+t_{k+1}}-1} \frac{1}{e^{t_{3}+\dots+t_{k+1}}-1} \cdots \frac{1}{e^{t_{k+1}}-1} \,dt_{2}\cdots dt_{k+1} \\
&=\Gamma(s)\zeta_{MT}(\{2\}^{j};1)\zeta(\{1\}^{k-1},s)-\Gamma(s)\zeta_{MT}(\{2\}^{j};2)\zeta(\{1\}^{k-2},s)\\
&\quad+\int_{0}^{\infty} \cdots \int_{0}^{\infty} t_{k+1}^{s-1}   {\rm Li}_{\{2\}^{j};2} (1-e^{-(t_{3}+\dots+t_{k+1})}) \\
&\quad\times \frac{1}{e^{t_{3}+\dots+t_{k+1}}-1} \frac{1}{e^{t_{4}+\dots+t_{k+1}}-1} \cdots \frac{1}{e^{t_{k+1}}-1} \,dt_{3}\cdots dt_{k+1} \\
&=\cdots \\
&=\Gamma(s)\sum_{i=1}^{k}(-1)^{i-1}\zeta_{MT}(\{2\}^{j};i)\zeta(\{1\}^{k-i},s) \\
&\quad+(-1)^{k}\int_{0}^{\infty} t_{k+1}^{s-1} {\rm Li}_{\{2\}^{j};k}(1-e^{-t_{k+1}})\frac{1}{e^{t_{k+1}}-1}\,dt_{k+1}.
\end{align*} 
Therefore, we have
\begin{align*}
J&=\Gamma(s)\zeta(2)^{l} \zeta(\{1\}^{k},s)+\Gamma(s)\sum_{j=1}^{l} \binom{l}{j} \zeta(2)^{l-j} (-1)^{j} \\
&\quad\times \Biggl( \sum_{i=1}^{k}(-1)^{i-1}\zeta_{MT}(\{2\}^{j};i)\zeta(\{1\}^{k-i},s)+(-1)^{k} \xi_{MT}((\{2\}^{j};k);s) \Biggr) \, .
\end{align*}

The second calculation is to use the polynomial expansion and the equation (\ref{eq:Mintrep}). By symmetry of $u_{1},\dots,u_{l}$, we have
\begin{align*}
J&=\sum_{a+b_{1}+\dots+b_{k+1}=l}\frac{l!}{a!b_{1}!\cdots b_{k+1}!}\int_{0}^{\infty} \cdots \int_{0}^{\infty} t_{k+1}^{s-1} u_{1} \cdots u_{a} t_{1}^{b_{1}}\cdots t_{k+1}^{b_{k+1}} \\
&\quad\times \left(\prod_{i=1}^{l}\frac{1}{e^{u_{i}+t_{1}+\dots+t_{k+1}}-1} \right)\left(\prod_{i=1}^{k+1}\frac{1}{e^{t_{i}+\dots+t_{k+1}}-1}\right) \,du_{1}\cdots du_{l}\,dt_{1}\cdots dt_{k+1} \\
&=\sum_{a+b_{1}+\dots+b_{k+1}=l}\frac{l!\Gamma(b_{k+1}+s)}{a!b_{k+1}!}\zeta(0,\{2\}^{a},\{1\}^{l-a};1+b_{1},\dots,1+b_{k},b_{k+1}+s).
\end{align*}
By the first and second calculations, we obtain the desired identity for $\Re(s) > 1$.
By the analytic continuation of the EZ-type MZF, the generalized Ito zeta function (\cite[Theorem 14]{Ito}) and the Miyagawa-type MZF, we obtain the stated theorem.
\end{proof}
Note that $J$ was calculated by Arakawa and Kaneko when $l=1$, and Ito when $k=0$ (see the proof of \cite[Theorem 6.(ii)]{ArakawaKaneko} and \cite[Theorem 8]{Ito}). Therefore, we may regard this proof as a fusion of the method of Arakawa and Kaneko, and the method of Ito.

\section{Generalized Mordell-Tornheim multiple zeta function}\label{se:Uzeta}
We will generalize Theorem \ref{th:IZfrel} and Theorem \ref{th:AKZofMfre} in the next section. For this purpose, we need to introduce a new class of the multiple zeta function (\ref{eq:defU}). In this section, we show several propositions on the function (\ref{eq:defU}).
\begin{remark}\label{re:Uzeta}
The function (\ref{eq:defU}) contains the EZ-type MZF, the MT-type MZF and the Miyagawa-type MZF as special cases. For example,
\[\zeta_{MT}((s_{1}, \dots, s_{j}); (s_{j+1}, \dots, s_{n})) = \zeta(s_{1}, \dots, s_{j-1}, s_{j}+s_{j+1}, s_{j+2}, \dots, s_{n}),\]
\[\zeta_{MT}((s_{1}), \dots, (s_{r}); (s_{r+1})) = \zeta_{MT}(s_{1}, \dots, s_{r}; s_{r+1}),\]
and
\[\zeta_{MT}((s_{1}), \dots, (s_{j}); (s_{j+1}, \dots, s_{r+1})) = \hat{\zeta}_{MT,j,r} (s_{1},\dots,s_{j};s_{j+1},\dots,s_{r+1}).\]
Therefore, for the case $n_{i} = 1\ (1 \le i \le r)$ and $\mathbf{s}_{r+1}=(s_{r+1,1},\dots,s_{r+1,n_{r+1}})$, we may omit the parentheses.
For example, we write $\zeta_{MT}((s_{1,1}), (s_{2,1}, s_{2,2});(s_{3,1}, s_{3,2}))$ as $\zeta_{MT}(s_{1,1}, (s_{2,1}, s_{2,2}); s_{3,1}, s_{3,2})$.
Under this notation, the Miyagawa-type MZF is written as
\[\hat{\zeta}_{MT,r,r+n_{r+1}-1} (s_{1},\dots,s_{r};s_{r+1,1},\dots,s_{r+1,n_{r+1}})=\zeta_{MT} (s_{1},\dots,s_{r};\mathbf{s}_{r+1}).\]
\end{remark}

\begin{proposition}\label{pr:Uconv}
If one of the following conditions is satisfied, then the function $(\ref{eq:defU})$  conveges absolutely.
\begin{enumerate}[{\rm (i)}]
\item\begin{align*}
\Re(s_{i,j}) &\ge 1 &&(1 \le i \le r,1 \le j \le n_{r}), \\
\sum_{i=0}^{k} \Re(s_{r+1,n_{r+1} -i}) &> k + 1 &&(0 \le k \le n_{r+1} - 2),\\
\sum_{i=0}^{n_{r+1} -1} \Re(s_{r+1,n_{r+1} -i})  &> n_{r+1} -1.
\end{align*}
\item\begin{align*}
\mathbf{s_{1}} &= (0),\\
\Re(s_{i,j}) &\ge 1 &&(2 \le i \le r,1 \le j \le n_{r}),\\
\sum_{i=0}^{k} \Re(s_{r+1,n_{r+1} -i}) &> k + 1 &&(0 \le k \le n_{r+1} - 1).
\end{align*}
\end{enumerate}
\end{proposition}
Note that these conditions are not a necessary condition  for absolute convergence. However, we mainly deal with the case when (ii) is satisfied.
\begin{proof}
The series $\sum_{n=1}^{\infty} \frac{1}{(N+n)^\sigma}\ (N>0)$ converges only when $\sigma>1$, and

\begin{equation*}
\sum_{n=1}^{\infty} \frac{1}{(N+n)^\sigma} \le \frac{1}{(\sigma - 1)N^{\sigma - 1}}.
\end{equation*}
Let $\Re(\mathbf{s}_{i}) = (\Re(s_{i,1}), \dots,\Re(s_{i,n_{i}}))\ (1 \le i \le r+1)$ and
\[\mathbf{s}^{(k)} = (\Re(s_{r+1,1}), \dots  ,  \Re(s_{r+1,n_{r+1}-k-1}), \sum_{i=0}^{k}\Re(s_{r+1,n_{r+1}-i} )  - k).\]
\begin{enumerate}[{\rm (i)}]
\item 
We have
\begin{align*}
&\zeta_{MT}(\Re(\mathbf{s}_{1}),\Re(\mathbf{s}_{2}),\dots,\Re(\mathbf{s}_{r});\Re(\mathbf{s}_{r+1})) \\
&\le \zeta_{MT}((\{1\}^{n_{1}}),\dots,(\{1\}^{n_{r}});\Re(\mathbf{s}_{r+1})) \\
&\ll \zeta_{MT}((\{1\}^{n_{1}}),\dots,(\{1\}^{n_{r}}); \mathbf{s}^{(1)}) \\
&\ll \zeta_{MT}((\{1\}^{n_{1}}),\dots,(\{1\}^{n_{r}}); \mathbf{s}^{(2)}) \\
&\ll \cdots \\
&\ll \zeta_{MT}((\{1\}^{n_{1}}),\dots,(\{1\}^{n_{r}});\sum_{i=0}^{n_{r+1}-1} \Re(s_{r+1,n_{r+1} -i}) -n_{r+1} + 1),
\end{align*}
where the implicit constants of $\ll$ depend on $\mathbf{s}_{r+1} =(s_{r+1,1}, \dots, s_{r+1,n_{r+1}})$.
Let $R=(\sum_{i=0}^{n_{r+1}-1} \Re(s_{r+1,n_{r+1} -i})-n_{r+1} + 1) / r $. Then $R>0$ and 
\[\frac{1}{(\sum_{v=1}^{r} m_{v,n_{v}})^{\sum_{i=0}^{n_{r+1}-1} \Re(s_{r+1,n_{r+1} -i})-n_{r+1} + 1}} \le \frac{1}{\prod_{v=1}^{r} m_{v,n_{v}}^{R}}.\]
Hence we have
\begin{align*}
&\zeta_{MT}(\Re(\mathbf{s}_{1}),\Re(\mathbf{s}_{2}),\dots,\Re(\mathbf{s}_{r});\Re(\mathbf{s}_{r+1})) \\
&\ll \zeta(\{1\}^{n_{1}-1},1+R) \zeta(\{1\}^{n_{2}-1},1+R) \cdots \zeta(\{1\}^{n_{r}-1},1+R).
\end{align*}
This completes the proof for the case (i).
\item In the same way as that of (i), we obtain
\begin{align*}
&\zeta_{MT}(0, \Re(\mathbf{s}_{2}),\dots,\Re(\mathbf{s}_{r});\Re(\mathbf{s}_{r+1})) \\
&\ll \zeta_{MT}(0,(\{1\}^{n_{2}}),\dots,(\{1\}^{n_{r}});\sum_{i=0}^{n_{r+1}-1} \Re(s_{r+1,n_{r+1} -i}) -n_{r+1} + 1).
\end{align*}
Let $\varepsilon > 0$ such that $\sum_{i=0}^{n_{r+1}-1} \Re(s_{r+1,n_{r+1} -i}) -n_{r+1} > \varepsilon$ and let \[R=\frac{\sum_{i=0}^{n_{r+1}-1} \Re(s_{r+1,n_{r+1} -i})-n_{r+1} -\varepsilon }{ r-1} \ (> 0).\]
By using
\[\frac{1}{(\sum_{v=1}^{r} m_{v,n_{v}})^{\sum_{i=0}^{n_{r+1}-1} \Re(s_{r+1,n_{r+1} -i})-n_{r+1} + 1}} \le \frac{1}{(m_{1,1})^{1+\varepsilon} \prod_{v=2}^{r} m_{v,n_{v}}^{R}},\]
we have
\begin{align*}
&\zeta_{MT}(0, \Re(\mathbf{s}_{2}),\dots,\Re(\mathbf{s}_{r});\Re(\mathbf{s}_{r+1})) \\
&\ll \zeta(1+\varepsilon) \zeta(\{1\}^{n_{2}-1},1+R) \cdots \zeta(\{1\}^{n_{r}-1},1+R).
\end{align*}
\end{enumerate}
This completes the proof.
\end{proof}
\begin{remark}
Let \[A={\begin{tikzpicture}[thick, baseline=-32mm]
\def\haba{0.9}

\node (1n_1) at (0,0) {$1$};
\node (2n_1) at ($(1n_1)+(0,-\haba)$) {$1$};
\node (3n_1) at ($(1n_1)+(\haba,-\haba)$) {$1$};
\node (4n_1) at ($(1n_1)+(7*\haba,0)$) {$$};

\node (1n_2) at ($(3n_1)+(0.5*\haba,-0.5*\haba)$) {$1$};
\node (2n_2) at ($(0,-\haba)+(1n_2)$) {$1$};
\node (3n_2) at ($(\haba,-\haba)+(1n_2)$){$1$};

\node (1n_r) at ($(3n_2)+2*(\haba,-\haba)$) {$1$};
\node (2n_r) at ($(0,-\haba)+(1n_r)$) {$1$};
\node (3n_r) at ($(\haba,-\haba)+(1n_r)$){$1$};
\coordinate (4n_r) at ($(\haba,0)+(1n_r)$);

\node (11n_r+1) at ($(1n_1)+(0,-6*\haba)$) {$1$};
\node (21n_r+1) at ($(11n_r+1)+(0,-0.5*\haba)$) {$1$};
\node (31n_r+1) at ($(21n_r+1)+(0,-\haba)$) {$1$};
\node (12n_r+1) at ($(3n_r)+(0,-0.5*\haba)$){$1$};
\node (22n_r+1) at ($(12n_r+1)+(0,-0.5*\haba)$){$1$};
\node (23n_r+1) at ($(12n_r+1)+(0.5*\haba,-0.5*\haba)$){$1$};
\node (32n_r+1) at ($(22n_r+1)+(0,-\haba)$){$1$};
\node (33n_r+1) at ($(23n_r+1)+(\haba,-\haba)$){$1$};
\coordinate (13n_r+1) at ($(1.5*\haba,0)+(12n_r+1)$);

\node (00) at ($(1n_1)+(0.5*\haba, -4.5*\haba)$){\fontsize{30pt}{0pt}\selectfont $0$};
\node (01) at ($(1n_1)+(6*\haba,-\haba)$){\fontsize{30pt}{0pt}\selectfont $0$};

\def\mdot{\node [circle, draw, black, fill=black,
inner sep=0pt, minimum width=0.1mm, label distance=0.5mm]}
\mdot (d1n_1) at ($(1n_1)!0.35!(3n_1)$) {};
\mdot (d2n_1) at ($(1n_1)!0.5!(3n_1)$) {};
\mdot (d3n_1) at ($(1n_1)!0.65!(3n_1)$) {};
\mdot (d4n_1) at ($(2n_1)!0.35!(3n_1)$) {};
\mdot (d5n_1) at ($(2n_1)!0.5!(3n_1)$) {};
\mdot (d6n_1) at ($(2n_1)!0.65!(3n_1)$) {};
\mdot (d7n_1) at ($(1n_1)!0.35!(2n_1)$) {};
\mdot (d8n_1) at ($(1n_1)!0.5!(2n_1)$) {};
\mdot (d9n_1) at ($(1n_1)!0.65!(2n_1)$) {};

\mdot (d1n_2) at ($(1n_2)!0.35!(3n_2)$) {};
\mdot (d2n_2) at ($(1n_2)!0.5!(3n_2)$) {};
\mdot (d3n_2) at ($(1n_2)!0.65!(3n_2)$) {};
\mdot (d4n_2) at ($(2n_2)!0.35!(3n_2)$) {};
\mdot (d5n_2) at ($(2n_2)!0.5!(3n_2)$) {};
\mdot (d6n_2) at ($(2n_2)!0.65!(3n_2)$) {};
\mdot (d7n_2) at ($(1n_2)!0.35!(2n_2)$) {};
\mdot (d8n_2) at ($(1n_2)!0.5!(2n_2)$) {};
\mdot (d9n_2) at ($(1n_2)!0.65!(2n_2)$) {};

\mdot (d1n_r) at ($(3n_2)!0.3!(1n_r)$) {};
\mdot (d2n_r) at ($(3n_2)!0.5!(1n_r)$) {};
\mdot (d3n_r) at ($(3n_2)!0.7!(1n_r)$) {};

\mdot (d1n_r) at ($(1n_r)!0.35!(3n_r)$) {};
\mdot (d2n_r) at ($(1n_r)!0.5!(3n_r)$) {};
\mdot (d3n_r) at ($(1n_r)!0.65!(3n_r)$) {};
\mdot (d4n_r) at ($(2n_r)!0.35!(3n_r)$) {};
\mdot (d5n_r) at ($(2n_r)!0.5!(3n_r)$) {};
\mdot (d6n_r) at ($(2n_r)!0.65!(3n_r)$) {};
\mdot (d7n_r) at ($(1n_r)!0.35!(2n_r)$) {};
\mdot (d8n_r) at ($(1n_r)!0.5!(2n_r)$) {};
\mdot (d9n_r) at ($(1n_r)!0.65!(2n_r)$) {};

\mdot (d1n_r+1) at ($(22n_r+1)!0.35!(32n_r+1)$) {};
\mdot (d1n_r+1) at ($(22n_r+1)!0.5!(32n_r+1)$) {};
\mdot (d3n_r+1) at ($(22n_r+1)!0.65!(32n_r+1)$) {};
\mdot (d4n_r+1) at ($(23n_r+1)!0.35!(33n_r+1)$) {};
\mdot (d5n_r+1) at ($(23n_r+1)!0.5!(33n_r+1)$) {};
\mdot (d6n_r+1) at ($(23n_r+1)!0.65!(33n_r+1)$) {};
\mdot (d7n_r+1) at ($(32n_r+1)!0.3!(33n_r+1)$) {};
\mdot (d8n_r+1) at ($(32n_r+1)!0.5!(33n_r+1)$) {};
\mdot (d9n_r+1) at ($(32n_r+1)!0.7!(33n_r+1)$) {};

\mdot (dv1) at ($(21n_r+1)!0.35!(31n_r+1)$) {};
\mdot (dv2) at ($(21n_r+1)!0.5!(31n_r+1)$) {};
\mdot (dv3) at ($(21n_r+1)!0.65!(31n_r+1)$) {};

\mdot (d11) at ($(11n_r+1)!0.1!(12n_r+1)$) {};
\mdot (d12) at ($(11n_r+1)!0.2!(12n_r+1)$) {};
\mdot (d13) at ($(11n_r+1)!0.3!(12n_r+1)$) {};
\mdot (d14) at ($(11n_r+1)!0.4!(12n_r+1)$) {};
\mdot (d15) at ($(11n_r+1)!0.5!(12n_r+1)$) {};
\mdot (d16) at ($(11n_r+1)!0.6!(12n_r+1)$) {};
\mdot (d17) at ($(11n_r+1)!0.7!(12n_r+1)$) {};
\mdot (d18) at ($(11n_r+1)!0.8!(12n_r+1)$) {};
\mdot (d19) at ($(11n_r+1)!0.9!(12n_r+1)$) {};

\mdot (d21) at ($(21n_r+1)!0.1!(22n_r+1)$) {};
\mdot (d22) at ($(21n_r+1)!0.2!(22n_r+1)$) {};
\mdot (d23) at ($(21n_r+1)!0.3!(22n_r+1)$) {};
\mdot (d24) at ($(21n_r+1)!0.4!(22n_r+1)$) {};
\mdot (d25) at ($(21n_r+1)!0.5!(22n_r+1)$) {};
\mdot (d26) at ($(21n_r+1)!0.6!(22n_r+1)$) {};
\mdot (d27) at ($(21n_r+1)!0.7!(22n_r+1)$) {};
\mdot (d28) at ($(21n_r+1)!0.8!(22n_r+1)$) {};
\mdot (d29) at ($(21n_r+1)!0.9!(22n_r+1)$) {};

\mdot (d31) at ($(31n_r+1)!0.1!(32n_r+1)$) {};
\mdot (d32) at ($(31n_r+1)!0.2!(32n_r+1)$) {};
\mdot (d33) at ($(31n_r+1)!0.3!(32n_r+1)$) {};
\mdot (d34) at ($(31n_r+1)!0.4!(32n_r+1)$) {};
\mdot (d35) at ($(31n_r+1)!0.5!(32n_r+1)$) {};
\mdot (d36) at ($(31n_r+1)!0.6!(32n_r+1)$) {};
\mdot (d37) at ($(31n_r+1)!0.7!(32n_r+1)$) {};
\mdot (d38) at ($(31n_r+1)!0.8!(32n_r+1)$) {};
\mdot (d39) at ($(31n_r+1)!0.9!(32n_r+1)$) {};

\draw[decorate, decoration={brace, mirror, amplitude = \BRACEAMPLITUDE}] ($(1n_1)+(-0.1,0.15)$) -- ($(2n_1)-(0.1,0.15)$) node[xshift = -4.5-4.5mm, yshift = 0mm, pos = 0.5] {\INDEX_SIZE$n_{1}$};
\draw[decorate, decoration={brace, mirror, amplitude = \BRACEAMPLITUDE}] ($(2n_1)+(-0.05,-0.2)$) -- ($(3n_1)+(0.05,-0.2)$) node[xshift = 0mm, yshift = -4.5mm, pos = 0.5] {\INDEX_SIZE$n_{1}$};

\draw[decorate, decoration={brace, mirror, amplitude = \BRACEAMPLITUDE}] ($(1n_2)+(-0.1,0.15)$) -- ($(2n_2)-(0.1,0.15)$) node[xshift = -4.5mm, yshift = 0mm, pos = 0.5] {\INDEX_SIZE$n_{2}$};
\draw[decorate, decoration={brace, mirror, amplitude = \BRACEAMPLITUDE}] ($(2n_2)+(-0.05,-0.2)$) -- ($(3n_2)+(0.05,-0.2)$) node[xshift = 0mm, yshift = -4.5mm, pos = 0.5] {\INDEX_SIZE$n_{2}$};

\draw[decorate, decoration={brace, mirror, amplitude = \BRACEAMPLITUDE}] ($(1n_r)+(-0.1,0.15)$) -- ($(2n_r)+(-0.1,-0.15)$) node[xshift = -5mm, yshift = 0mm, pos = 0.5] {\INDEX_SIZE$n_{r}$};
\draw[decorate, decoration={brace, mirror, aspect=0.3, amplitude = \BRACEAMPLITUDE}] ($(2n_r)+(-0.05,-0.2)$) -- ($(3n_r)+(0.05,-0.2)$) node[xshift = -2mm, yshift = -4.5mm, pos = 0.5] {\INDEX_SIZE$n_{r}$};

\draw[decorate, decoration={brace, mirror, aspect=0.6, amplitude = \BRACEAMPLITUDE}] ($(12n_r+1)+(-0.1,0.15)$) -- ($(32n_r+1)-(0.1,0.15)$) node[xshift = -6.5-4mm, yshift = -2mm, pos = 0.5] {\INDEX_SIZE$n_{r+1}$};
\draw[decorate, decoration={brace,mirror, amplitude = \BRACEAMPLITUDE}] ($(32n_r+1)+(-0.05,-0.2)$) -- ($(33n_r+1)+(0.05,-0.2)$) node[xshift = 0mm, yshift = -4.5mm, pos = 0.5] {\INDEX_SIZE$n_{r+1}$};

\draw [bend right,distance=35] ($(1n_1)+(-0.15,0.3)$) to ($(31n_r+1)+(-0.15,-0.3)$); 
\draw [bend left,distance=35] ($(4n_1)+(0.15,0.3)$) to ($(33n_r+1)+(0.15,-0.3)$); 
\end{tikzpicture}}
.\]
Then we have 
\[\zeta_{MT}(\mathbf{s}_{1},\mathbf{s}_{2},\dots,\mathbf{s}_{r};\mathbf{s}_{r+1}) = \zeta(\mathbf{s}_{1},\mathbf{s}_{2},\dots,\mathbf{s}_{r},\mathbf{s}_{r+1};A).\]
Therefore, the function (\ref{eq:defU}) is included as a special case of the function (\ref{eq:Mzeta}). From Remark \ref{re:matrixnotinuas}, the function (\ref{eq:defU}) can be continued meromorphically to the whole space $\mathbb{C}^{n_{1}+\dots+n_{r+1}}$. Moreover, by Lemma \ref{le:Aintrep}, especially when the condition (i) of Proposition~\ref{pr:Uconv} is satisfied, we have
\begin{align*}
&\zeta_{MT}(\mathbf{s}_{1},\dots,\mathbf{s}_{r};\mathbf{s}_{r+1}) \\
&=\frac{1}{\prod_{i=1}^{r+1}\prod_{j=1}^{n_{i}}\Gamma(s_{i,j})}\int_{0}^{\infty}\cdots\int_{0}^{\infty}\prod_{i=1}^{r} \prod_{j=1}^{n_{i}} \frac{t_{i,j}^{s_{i,j}-1}\,dt_{i,j}}{e^{t_{i,j}+\dots +t_{i,n_{i}} + t_{r+1,1} + \dots +t_{r+1,n_{r+1}}} -1} \\
&\quad\times\,t_{r+1,1}^{s_{r+1,1} -1}\,dt_{r+1,1} \,\prod_{u=2}^{n_{r+1}} \frac{t_{r+1,u}^{s_{r+1,u}-1}\,dt_{r+1,u}}{e^{t_{r+1,u} + \dots + t_{r+1,n_{r+1}}} -1}.
\end{align*}
Let \[A={\begin{tikzpicture}[thick, baseline=-26mm]
\def\haba{0.9}
\node (1n_2) at (0,0) {$1$};
\node (2n_2) at ($(0,-\haba)+(1n_2)$) {$1$};
\node (3n_2) at ($(\haba,-\haba)+(1n_2)$){$1$};
\node (4n_2) at ($(1n_2)+(5.5*\haba,0)$) {$$};

\node (1n_r) at ($(3n_2)+2*(\haba,-\haba)$) {$1$};
\node (2n_r) at ($(0,-\haba)+(1n_r)$) {$1$};
\node (3n_r) at ($(\haba,-\haba)+(1n_r)$){$1$};
\coordinate (4n_r) at ($(\haba,0)+(1n_r)$);

\node (11n_r+1) at ($(1n_2)+(-0.5*\haba,-4.5*\haba)$) {$1$};
\node (21n_r+1) at ($(11n_r+1)+(0,-0.5*\haba)$) {$1$};
\node (31n_r+1) at ($(21n_r+1)+(0,-\haba)$) {$1$};
\node (12n_r+1) at ($(3n_r)+(0,-0.5*\haba)$){$1$};
\node (22n_r+1) at ($(12n_r+1)+(0,-0.5*\haba)$){$1$};
\node (23n_r+1) at ($(12n_r+1)+(0.5*\haba,-0.5*\haba)$){$1$};
\node (32n_r+1) at ($(22n_r+1)+(0,-\haba)$){$1$};
\node (33n_r+1) at ($(23n_r+1)+(\haba,-\haba)$){$1$};
\coordinate (13n_r+1) at ($(1.5*\haba,0)+(12n_r+1)$);

\node (01n_2) at ($(1n_1)+(-0.5*\haba,0)$) {$0$};
\node (02n_2) at ($(11n_r+1)+(0,0.5*\haba)$) {$0$};

\node (00) at ($(1n_2)+(0.5*\haba, -3.3*\haba)$){\fontsize{30pt}{0pt}\selectfont $0$};
\node (01) at ($(1n_2)+(4.5*\haba,-1*\haba)$){\fontsize{30pt}{0pt}\selectfont $0$};

\def\mdot{\node [circle, draw, black, fill=black,
inner sep=0pt, minimum width=0.1mm, label distance=0.5mm]}

\mdot (d1n_2) at ($(1n_2)!0.35!(3n_2)$) {};
\mdot (d2n_2) at ($(1n_2)!0.5!(3n_2)$) {};
\mdot (d3n_2) at ($(1n_2)!0.65!(3n_2)$) {};
\mdot (d4n_2) at ($(2n_2)!0.35!(3n_2)$) {};
\mdot (d5n_2) at ($(2n_2)!0.5!(3n_2)$) {};
\mdot (d6n_2) at ($(2n_2)!0.65!(3n_2)$) {};
\mdot (d7n_2) at ($(1n_2)!0.35!(2n_2)$) {};
\mdot (d8n_2) at ($(1n_2)!0.5!(2n_2)$) {};
\mdot (d9n_2) at ($(1n_2)!0.65!(2n_2)$) {};

\mdot (d1n_r) at ($(3n_2)!0.3!(1n_r)$) {};
\mdot (d2n_r) at ($(3n_2)!0.5!(1n_r)$) {};
\mdot (d3n_r) at ($(3n_2)!0.7!(1n_r)$) {};

\mdot (d1n_r) at ($(1n_r)!0.35!(3n_r)$) {};
\mdot (d2n_r) at ($(1n_r)!0.5!(3n_r)$) {};
\mdot (d3n_r) at ($(1n_r)!0.65!(3n_r)$) {};
\mdot (d4n_r) at ($(2n_r)!0.35!(3n_r)$) {};
\mdot (d5n_r) at ($(2n_r)!0.5!(3n_r)$) {};
\mdot (d6n_r) at ($(2n_r)!0.65!(3n_r)$) {};
\mdot (d7n_r) at ($(1n_r)!0.35!(2n_r)$) {};
\mdot (d8n_r) at ($(1n_r)!0.5!(2n_r)$) {};
\mdot (d9n_r) at ($(1n_r)!0.65!(2n_r)$) {};

\mdot (d1n_r+1) at ($(22n_r+1)!0.35!(32n_r+1)$) {};
\mdot (d1n_r+1) at ($(22n_r+1)!0.5!(32n_r+1)$) {};
\mdot (d3n_r+1) at ($(22n_r+1)!0.65!(32n_r+1)$) {};
\mdot (d4n_r+1) at ($(23n_r+1)!0.35!(33n_r+1)$) {};
\mdot (d5n_r+1) at ($(23n_r+1)!0.5!(33n_r+1)$) {};
\mdot (d6n_r+1) at ($(23n_r+1)!0.65!(33n_r+1)$) {};
\mdot (d7n_r+1) at ($(32n_r+1)!0.3!(33n_r+1)$) {};
\mdot (d8n_r+1) at ($(32n_r+1)!0.5!(33n_r+1)$) {};
\mdot (d9n_r+1) at ($(32n_r+1)!0.7!(33n_r+1)$) {};

\mdot (d11) at ($(01n_2)!0.1!(02n_2)$) {};
\mdot (d12) at ($(01n_2)!0.2!(02n_2)$) {};
\mdot (d13) at ($(01n_2)!0.3!(02n_2)$) {};
\mdot (d14) at ($(01n_2)!0.4!(02n_2)$) {};
\mdot (d15) at ($(01n_2)!0.5!(02n_2)$) {};
\mdot (d16) at ($(01n_2)!0.6!(02n_2)$) {};
\mdot (d17) at ($(01n_2)!0.7!(02n_2)$) {};
\mdot (d18) at ($(01n_2)!0.8!(02n_2)$) {};
\mdot (d19) at ($(01n_2)!0.9!(02n_2)$) {};

\mdot (dv1) at ($(21n_r+1)!0.35!(31n_r+1)$) {};
\mdot (dv2) at ($(21n_r+1)!0.5!(31n_r+1)$) {};
\mdot (dv3) at ($(21n_r+1)!0.65!(31n_r+1)$) {};

\mdot (d11) at ($(11n_r+1)!0.1!(12n_r+1)$) {};
\mdot (d12) at ($(11n_r+1)!0.2!(12n_r+1)$) {};
\mdot (d13) at ($(11n_r+1)!0.3!(12n_r+1)$) {};
\mdot (d14) at ($(11n_r+1)!0.4!(12n_r+1)$) {};
\mdot (d15) at ($(11n_r+1)!0.5!(12n_r+1)$) {};
\mdot (d16) at ($(11n_r+1)!0.6!(12n_r+1)$) {};
\mdot (d17) at ($(11n_r+1)!0.7!(12n_r+1)$) {};
\mdot (d18) at ($(11n_r+1)!0.8!(12n_r+1)$) {};
\mdot (d19) at ($(11n_r+1)!0.9!(12n_r+1)$) {};

\mdot (d21) at ($(21n_r+1)!0.1!(22n_r+1)$) {};
\mdot (d22) at ($(21n_r+1)!0.2!(22n_r+1)$) {};
\mdot (d23) at ($(21n_r+1)!0.3!(22n_r+1)$) {};
\mdot (d24) at ($(21n_r+1)!0.4!(22n_r+1)$) {};
\mdot (d25) at ($(21n_r+1)!0.5!(22n_r+1)$) {};
\mdot (d26) at ($(21n_r+1)!0.6!(22n_r+1)$) {};
\mdot (d27) at ($(21n_r+1)!0.7!(22n_r+1)$) {};
\mdot (d28) at ($(21n_r+1)!0.8!(22n_r+1)$) {};
\mdot (d29) at ($(21n_r+1)!0.9!(22n_r+1)$) {};

\mdot (d31) at ($(31n_r+1)!0.1!(32n_r+1)$) {};
\mdot (d32) at ($(31n_r+1)!0.2!(32n_r+1)$) {};
\mdot (d33) at ($(31n_r+1)!0.3!(32n_r+1)$) {};
\mdot (d34) at ($(31n_r+1)!0.4!(32n_r+1)$) {};
\mdot (d35) at ($(31n_r+1)!0.5!(32n_r+1)$) {};
\mdot (d36) at ($(31n_r+1)!0.6!(32n_r+1)$) {};
\mdot (d37) at ($(31n_r+1)!0.7!(32n_r+1)$) {};
\mdot (d38) at ($(31n_r+1)!0.8!(32n_r+1)$) {};
\mdot (d39) at ($(31n_r+1)!0.9!(32n_r+1)$) {};

\draw[decorate, decoration={brace, mirror, aspect=0.55, amplitude = \BRACEAMPLITUDE}] ($(1n_2)+(-0.1,0.15)$) -- ($(2n_2)-(0.1,0.15)$) node[xshift = -5.4mm, yshift = -1mm, pos = 0.5] {\INDEX_SIZE$n_{2}$};
\draw[decorate, decoration={brace, mirror, amplitude = \BRACEAMPLITUDE}] ($(2n_2)+(-0.05,-0.2)$) -- ($(3n_2)+(0.05,-0.2)$) node[xshift = 0mm, yshift = -4.5mm, pos = 0.5] {\INDEX_SIZE$n_{2}$};

\draw[decorate, decoration={brace, mirror, amplitude = \BRACEAMPLITUDE}] ($(1n_r)+(-0.1,0.15)$) -- ($(2n_r)+(-0.1,-0.15)$) node[xshift = -5mm, yshift = 0mm, pos = 0.5] {\INDEX_SIZE$n_{r}$};
\draw[decorate, decoration={brace, mirror, aspect=0.3, amplitude = \BRACEAMPLITUDE}] ($(2n_r)+(-0.05,-0.2)$) -- ($(3n_r)+(0.05,-0.2)$) node[xshift = -2mm, yshift = -5mm, pos = 0.5] {\INDEX_SIZE$n_{r}$};

\draw[decorate, decoration={brace, mirror, aspect=0.6, amplitude = \BRACEAMPLITUDE}] ($(12n_r+1)+(-0.1,0.15)$) -- ($(32n_r+1)-(0.1,0.15)$) node[xshift = -6.5-4mm, yshift = -2mm, pos = 0.5] {\INDEX_SIZE$n_{r+1}$};
\draw[decorate, decoration={brace,mirror, amplitude = \BRACEAMPLITUDE}] ($(32n_r+1)+(-0.05,-0.2)$) -- ($(33n_r+1)+(0.05,-0.2)$) node[xshift = 0mm, yshift = -4.5mm, pos = 0.5] {\INDEX_SIZE$n_{r+1}$};

\draw [bend right,distance=35] ($(01n_2)+(-0.15,0.3)$) to ($(31n_r+1)+(-0.15,-0.3)$); 
\draw [bend left,distance=35] ($(4n_2)+(0.15,0.3)$) to ($(33n_r+1)+(0.15,-0.3)$); 
\end{tikzpicture}}
.\]
Then we have
\[\zeta_{MT}(0,\mathbf{s}_{2},\dots,\mathbf{s}_{r};\mathbf{s}_{r+1}) = \zeta(0,\mathbf{s}_{2},\dots,\mathbf{s}_{r},\mathbf{s}_{r+1};A).\]
Therefore, by Lemma \ref{le:Aintrep}, especially when the condition (ii) of Proposition \ref{pr:Uconv} is satisfied, we have
\begin{align}
&\zeta_{MT}(0,\mathbf{s}_{2},\dots,\mathbf{s}_{r};\mathbf{s}_{r+1}) \label{eq:Uintrep}\\
&=\frac{1}{\prod_{i=2}^{r+1}\prod_{j=1}^{n_{i}}\Gamma(s_{i,j})}\int_{0}^{\infty}\cdots\int_{0}^{\infty} \prod_{i=2}^{r} \prod_{j=1}^{n_{i}} \frac{t_{i,j}^{s_{i,j}-1}\,dt_{i,j}}{e^{t_{i,j}+\dots +t_{i,n_{i}} + t_{r+1,1} + \dots +t_{r+1,n_{r+1}}} -1} \nonumber \\
&\quad\times\,\prod_{u=1}^{n_{r+1}} \frac{t_{r+1,u}^{s_{r+1,u}-1}\,dt_{r+1,u}}{e^{t_{r+1,u} + \dots + t_{r+1,n_{r+1}}} -1}. \nonumber 
\end{align}
\end{remark}
\begin{remark}For $z \in \mathbb{C}$ with $|z| <1$, let
\begin{align*}
&{\rm Li}_{\mathbf{k}_{1},\dots,\mathbf{k}_{r};\mathbf{k}_{r+1}}(z) \\
&=\sum_{\substack{0<m_{1,1}<m_{1,2}<\dots<m_{1,n_{1}} \\ \vdots \\ 0<m_{r,1}<m_{r,2}<\dots<m_{r,n_{r}}}} \, \sum_{m_{r+1,1}=1,\dots ,m_{r+1,n_{r+1}-1}=1}^{\infty} \\
&\quad\frac{z^{\sum_{v=1}^{r} m_{v,n_{v}} +\sum_{w=1}^{n_{r+1}-1} m_{r+1,w}}}{\prod_{i=1}^{r} \prod_{j=1}^{n_{i}} m_{i,j}^{k_{i,j}} \prod_{u=1}^{n_{r+1}} (\sum_{v=1}^{r} m_{v,n_{v}} +\sum_{w=1}^{u-1} m_{r+1,w})^{k_{r+1,u}}}.\\
\end{align*}
Then we obtain
\begin{align*}
&\frac{d}{dz}{\rm {Li}}_{\mathbf{k}_{1},\dots,\mathbf{k}_{r};\mathbf{k}_{r+1}}(z)\\
&=\begin{cases}
\displaystyle
\frac{1}{z} {\rm Li}_{\mathbf{k}_{1},\dots,\mathbf{k}_{r};(\mathbf{k}_{r+1})_{-}}(z)
&({\rm if}\ k_{r+1,n_{r+1}} > 1),\\
\displaystyle
\frac{1}{1-z} {\rm Li}_{\mathbf{k}_{1},\dots,\mathbf{k}_{r};k_{r+1,1},\dots,k_{r+1,n_{r+1}-1}}(z)
&({\rm if}\ k_{r+1,n_{r+1}} = 1, n_{r+1} > 1),\\
\displaystyle
\frac{1}{z} \prod_{i=1}^{r}{\rm Li}_{\mathbf{k}_i}(z)
&({\rm if}\ k_{r+1,n_{r+1}} =1, n_{r+1} = 1).
\end{cases}
\end{align*}
Therefore, the special values of the GMT-type MZF are written as follows:
\begin{align*}
&\zeta_{MT}(\mathbf{k}_{1}, \dots, \mathbf{k}_{r}; \mathbf{k}_{r+1}) \\
&=\int_{0<t_{1}<\dots<t_{k_{r+1,1} +\dots+k_{r+1,n_{r+1}}}<1} \left(\prod_{i=2}^{k_{r+1,n_{r+1}}}\frac{1}{t_{k_{r+1,1} +\dots+k_{r+1,n_{r+1}-1} +i}}\right) \\
&\quad\times\frac{1}{1-t_{k_{r+1,1} +\dots+k_{r+1,n_{r+1}-1} +1}} \left(\prod_{i=2}^{k_{r+1,n_{r+1}-1}}\frac{1}{t_{k_{r+1,1} +\dots+k_{r+1,n_{r+1}-2} +i}}\right)\cdots\\
&\quad\dots \times\frac{1}{1-t_{k_{r+1,1}+k_{r+1,2}+1}} \left(\prod_{i=2}^{k_{r+1,2}}\frac{1}{t_{k_{r+1,1}+i}}\right)\\
&\quad\times\frac{1}{1-t_{k_{r+1,1}+1}} \left(\prod_{i=1}^{t_{k_{r+1,1}}}\frac{1}{t_{i}}\right)\prod_{i=1}^{r}{\rm Li}_{\mathbf{k}_i}(t_{1})\,dt_{1}\cdots dt_{k_{r+1,1}+\dots+k_{r+1,n_{r+1}}} \\
&= I\left(
\begin{tikzpicture}[thick, baseline=-55mm]
{
\Whitecircle (E0_0_0) at (0,0) {};
\Whitecircle (E0_0_1) at ($(E0_0_0) + (0, -\DOT_DISTANCE)$) {};
\drawvdots{E0_0_0}{E0_0_1};
\Blackcircle (E0_0_2) at ($(E0_0_1) + (0, -\LINE_DISTANCE)$) {};
\draw [line width = \LINEWIDTH] (E0_0_1) -- (E0_0_2);
\Blackcircle (E0_0_3) at ($(E0_0_2) + (0, -\DOT_DISTANCE)$) {};
\drawvdots{E0_0_2}{E0_0_3};
\Whitecircle (E0_0_4) at ($(E0_0_3) + (0, -\LINE_DISTANCE)$) {};
\draw [line width = \LINEWIDTH] (E0_0_3) -- (E0_0_4);
\Whitecircle (E0_0_5) at ($(E0_0_4) + (0, -\DOT_DISTANCE)$) {};
\drawvdots{E0_0_4}{E0_0_5};
\Blackcircle (E0_0_6) at ($(E0_0_5) + (0, -\LINE_DISTANCE)$) {};
\draw [line width = \LINEWIDTH] (E0_0_5) -- (E0_0_6);
\Whitecircle (E0_0_7) at ($(E0_0_6) + (0, -\LINE_DISTANCE)$) {};
\draw [line width = \LINEWIDTH] (E0_0_6) -- (E0_0_7);
\Whitecircle (E0_0_8) at ($(E0_0_7) + (0, -\DOT_DISTANCE)$) {};
\drawvdots{E0_0_7}{E0_0_8};
\def\branch{E0_0_8};
\Whitecircle (E1_0_0) at ($(\branch) + ( -2 *  \BRANCH_WIDTH, -\LINE_DISTANCE)$) {};
\draw  [line width = \LINEWIDTH](\branch) -- (E1_0_0);
\Whitecircle (E1_0_1) at ($(E1_0_0) + (0, - \DOT_DISTANCE)$) {};
\drawvdots{E1_0_0}{E1_0_1}
\Blackcircle (E1_0_2) at ($(E1_0_1) + (0, - \LINE_DISTANCE)$) {};
\draw [line width = \LINEWIDTH] (E1_0_1) -- (E1_0_2);
\Blackcircle (E1_0_3) at ($(E1_0_2) + (0, - \DOT_DISTANCE)$) {};
\drawvdots{E1_0_2}{E1_0_3}
\Whitecircle (E1_0_4) at ($(E1_0_3) + (0, -\LINE_DISTANCE)$) {};
\draw  [line width = \LINEWIDTH](E1_0_4) -- (E1_0_3);
\Whitecircle (E1_0_5) at ($(E1_0_4) + (0, - \DOT_DISTANCE)$) {};
\drawvdots{E1_0_4}{E1_0_5}
\Blackcircle (E1_0_6) at ($(E1_0_5) + (0, - \LINE_DISTANCE)$) {};
\draw [line width = \LINEWIDTH] (E1_0_5) -- (E1_0_6);
\Whitecircle (E3_0_0) at ($(\branch) + ( 2*\BRANCH_WIDTH, -\LINE_DISTANCE)$) {};
\draw [line width = \LINEWIDTH] (\branch) -- (E3_0_0);
\Whitecircle (E3_0_1) at ($(E3_0_0) + (0, - \DOT_DISTANCE)$) {};
\drawvdots{E3_0_0}{E3_0_1}
\Blackcircle (E3_0_2) at ($(E3_0_1) + (0, - \LINE_DISTANCE)$) {};
\draw [line width = \LINEWIDTH] (E3_0_1) -- (E3_0_2);
\Blackcircle (E3_0_3) at ($(E3_0_2) + (0, - \DOT_DISTANCE)$) {};
\drawvdots{E3_0_2}{E3_0_3}
\Whitecircle (E3_0_4) at ($(E3_0_3) + (0, -\LINE_DISTANCE)$) {};
\draw  [line width = \LINEWIDTH](E3_0_4) -- (E3_0_3);
\Whitecircle (E3_0_5) at ($(E3_0_4) + (0, - \DOT_DISTANCE)$) {};
\drawvdots{E3_0_4}{E3_0_5}
\Blackcircle (E3_0_6) at ($(E3_0_5) + (0, - \LINE_DISTANCE)$) {};
\draw [line width = \LINEWIDTH] (E3_0_5) -- (E3_0_6);
\Whitecircle (E2_0_0) at ($(E1_0_0)!\SECOND_EDGE_PROPOTION!(E3_0_0)$) {};
\draw  [line width = \LINEWIDTH] (\branch) -- (E2_0_0);
\Whitecircle (E2_0_1) at ($(E2_0_0) + (0, - \DOT_DISTANCE)$) {};
\drawvdots{E2_0_0}{E2_0_1}
\Blackcircle (E2_0_2) at ($(E2_0_1) + (0, - \LINE_DISTANCE)$) {};
\draw [line width = \LINEWIDTH] (E2_0_1) -- (E2_0_2);
\Blackcircle (E2_0_3) at ($(E2_0_2) + (0, - \DOT_DISTANCE)$) {};
\drawvdots{E2_0_2}{E2_0_3}
\Whitecircle (E2_0_4) at ($(E2_0_3) + (0, -\LINE_DISTANCE)$) {};
\draw  [line width = \LINEWIDTH](E2_0_4) -- (E2_0_3);
\Whitecircle (E2_0_5) at ($(E2_0_4) + (0, - \DOT_DISTANCE)$) {};
\drawvdots{E2_0_4}{E2_0_5}
\Blackcircle (E2_0_6) at ($(E2_0_5) + (0, - \LINE_DISTANCE)$) {};
\draw [line width = \LINEWIDTH] (E2_0_5) -- (E2_0_6);
\draw[decorate, decoration={brace, mirror, amplitude = \BRACEAMPLITUDE}]
($(E0_0_0) + (- \BRACEVSPACE, \BRACEHSPACE)$) -- ($(E0_0_1) + (- \BRACEVSPACE, - \BRACEHSPACE) $) node[xshift = -\BRACEXSIFT*1.7, yshift = 0mm, pos = 0.5] {\INDEX_SIZE$k_{r+1,n_{r+1}}-1$};
\draw[decorate, decoration={brace, mirror, amplitude = \BRACEAMPLITUDE}]
($(E0_0_4) + (- \BRACEVSPACE, \BRACEHSPACE)$) -- ($(E0_0_5) + (- \BRACEVSPACE, - \BRACEHSPACE) $) node[xshift = -\BRACEXSIFT*1.3, yshift = 0mm, pos = 0.5] {\INDEX_SIZE$k_{r+1,2}-1$};
\draw[decorate, decoration={brace, mirror, amplitude = \BRACEAMPLITUDE}]
($(E0_0_7) + (- \BRACEVSPACE, \BRACEHSPACE)$) -- ($(E0_0_8) + (- \BRACEVSPACE, - \BRACEHSPACE) $) node[xshift = -\BRACEXSIFT*1.3, yshift = 0mm, pos = 0.5] {\INDEX_SIZE$k_{r+1,1}$};
\draw[decorate, decoration={brace, mirror, amplitude = \BRACEAMPLITUDE}]
($(E1_0_0) + ( - \BRACEVSPACE, \BRACEHSPACE)$) -- ($(E1_0_1) + ( - \BRACEVSPACE, - \BRACEHSPACE) $) node[xshift = - 1.3*\BRACEXSIFT, yshift = 0mm, pos = 0.5] {\INDEX_SIZE$k_{1,n_{1}} - 1$};
\draw[decorate, decoration={brace, mirror, amplitude = \BRACEAMPLITUDE}]
($(E1_0_4) + ( - \BRACEVSPACE, \BRACEHSPACE)$) -- ($(E1_0_5) + ( - \BRACEVSPACE, - \BRACEHSPACE) $) node[xshift = - 1.3*\BRACEXSIFT, yshift = 0mm, pos = 0.5] {\INDEX_SIZE$k_{1,1} - 1$};
\draw[decorate, decoration={brace, amplitude = \BRACEAMPLITUDE}]
($(E2_0_0) + (\BRACEVSPACE, \BRACEHSPACE)$) -- ($(E2_0_1) + (\BRACEVSPACE, - \BRACEHSPACE) $) node[xshift = 1.3*\BRACEXSIFT, yshift = 0mm, pos = 0.5] {\INDEX_SIZE$k_{2,n_{2}} - 1$};
\draw[decorate, decoration={brace, amplitude = \BRACEAMPLITUDE}]
($(E2_0_4) + (\BRACEVSPACE, \BRACEHSPACE)$) -- ($(E2_0_5) + (\BRACEVSPACE, - \BRACEHSPACE) $) node[xshift = 1.3*\BRACEXSIFT, yshift = 0mm, pos = 0.5] {\INDEX_SIZE$k_{2,1} - 1$};
\draw[decorate, decoration={brace, amplitude = \BRACEAMPLITUDE}]
($(E3_0_0) + (\BRACEVSPACE, \BRACEHSPACE)$) -- ($(E3_0_1) + (\BRACEVSPACE, - \BRACEHSPACE) $) node[xshift = 1.3*\BRACEXSIFT, yshift = 0mm, pos = 0.5] {\INDEX_SIZE$k_{r,n_{r}} - 1$};
\draw[decorate, decoration={brace, amplitude = \BRACEAMPLITUDE}]
($(E3_0_4) + (\BRACEVSPACE, \BRACEHSPACE)$) -- ($(E3_0_5) + (\BRACEVSPACE, - \BRACEHSPACE) $) node[xshift = 1.3*\BRACEXSIFT, yshift = 0mm, pos = 0.5] {\INDEX_SIZE$k_{r,1} - 1$};
\tdot (d1) at ($(E2_0_5)!.5!(E2_0_6) + (1.8*\DOTSBRANK,  \DOTSPOSITION)$) {};
\tdot (d3) at ($(E3_0_5)!.5!(E3_0_6) + (- 1.8*\DOTSBRANK, \DOTSPOSITION)$){};
\tdot (d2) at ($(d1)!.5!(d3)$){};
}\end{tikzpicture}
\right).
\end{align*}
Therefore, GMT-type MZVs can be expressed as a sum of a finite number of EZ-type MZVs by \cite[Corollary 2.4]{Yamamoto}.
\end{remark}
\begin{proposition}
If the condition {\rm (ii)} of Proposition \ref{pr:Uconv} is satisfied and all entries of $\mathbf{s}_{i}\,(2 \le i \le r)$ are positive integers {\rm (}we put $\mathbf{s}_{i}=\mathbf{k}_{i} (= (k_{i,1}, \dots, k_{i, n_{i}})) \in \mathbb{N}^{n_{i}}\,(2 \le i \le r)${\rm )}, then $\zeta_{MT}(0,\mathbf{k}_{2},\dots,\mathbf{k}_{r};\mathbf{s}_{r+1})$ is expressed as a $\mathbb{Q}$-linear combination of EZ-type MZFs.
\end{proposition}
\begin{proof}
We have
\begin{align}
&\zeta_{MT}(0,\mathbf{k}_{2},\dots,\mathbf{k}_{r};\mathbf{s}_{r+1}) \nonumber \\
&=\sum_{\substack{0<m_{1,1} \\ 0<m_{2,1}<m_{2,2}<\dots<m_{2,n_{2}} \\ \vdots \\ 0<m_{r,1}<m_{r,2}<\dots<m_{r,n_{r}}}} \, \sum_{m_{r+1,1}=1,\dots ,m_{r+1,n_{r+1}-1}=1}^{\infty} \nonumber \\
&\quad\frac{1}{\prod_{i=2}^{r} \prod_{j=1}^{n_{i}} m_{i,j}^{k_{i,j}} \prod_{u=1}^{n_{r+1}} (\sum_{v=1}^{r} m_{v,n_{v}} +\sum_{w=1}^{u-1} m_{r+1,w})^{s_{r+1,u}}}  \nonumber\\
&=\sum_{\substack{0<m_{1,1} \\ 0<m_{2,1}<m_{2,2}<\dots<m_{2,n_{2}} \\ \vdots \\ 0<m_{r,1}<m_{r,2}<\dots<m_{r,n_{r}}}} \, \sum_{m_{r+1,1}=1,\dots ,m_{r+1,n_{r+1}-1}=1}^{\infty} \frac{1}{\prod_{i=2}^{r} \prod_{j=1}^{n_{i}}m_{i,j}^{k_{i,j}}}  \nonumber \\
&\quad\times\left(\prod_{u=1}^{n_{r+1}} \frac{1}{\Gamma(s_{r+1,u})} \int_{0}^{\infty} t_{u}^{s_{r+1,u} - 1} e^{-(\sum_{v=1}^{r} m_{v,n_{v}} +\sum_{w=1}^{u-1} m_{r+1,w})t_{u}}\,dt_{u}\right) \nonumber \\
&=\frac{1}{\prod_{u=1}^{n_{r+1}}\Gamma(s_{r+1,u})} \sum_{\substack{0<m_{2,1}<m_{2,2}<\dots<m_{2,n_{2}} \\ \vdots \\ 0<m_{r,1}<m_{r,2}<\dots<m_{r,n_{r}}}} \frac{1}{\prod_{i=2}^{r} \prod_{j=1}^{n_{i}}m_{i,j}^{k_{i,j}}}  \nonumber \\
&\quad\times\int_{0}^{\infty} \cdots \int_{0}^{\infty} t_{1}^{s_{r+1,1}-1} \cdots t_{n_{r+1}}^{s_{r+1,n_{r+1}}-1} e^{-(\sum_{v=2}^{r} m_{v,n_{v}})(t_{1}+\dots +t_{n_{r+1}})}  \nonumber \\
&\quad\times \frac{1}{e^{t_{1}+\dots+t_{n_{r+1}}} -1} \frac{1}{e^{t_{2}+\dots+t_{n_{r+1}}} -1} \cdots \frac{1}{e^{t_{n_{r+1}}} -1}\,dt_{1} \cdots dt_{n_{r+1}} \nonumber \\
&=\frac{1}{\prod_{u=1}^{n_{r+1}} \Gamma(s_{r+1,u})} \int_{0}^{\infty} \cdots \int_{0}^{\infty} t_{1}^{s_{r+1,1}-1} \cdots t_{n_{r+1}}^{s_{r+1,n_{r+1}}-1}   \label{int}  \\
&\quad\times \frac{\prod_{i=2}^{r} {\rm Li}_{\mathbf{k}_{i}}(e^{-(t_{1}+\dots+t_{n_{r+1}})})}{e^{t_{1}+\dots+t_{n_{r+1}}} -1} \frac{1}{e^{t_{2}+\dots+t_{n_{r+1}}} -1} \cdots \frac{1}{e^{t_{n_{r+1}}} -1}\,dt_{1} \cdots dt_{n_{r+1}}. \nonumber
\end{align}
By using the shuffle product formula for $\prod_{i=2}^{r} {\rm Li}_{\mathbf{k}_{i}}(e^{-(t_{1}+\dots+t_{n_{r+1}})})$, we find that (\ref{int}) is expressed as a $\mathbb{Q}$-linear combination of the form
\begin{align}
&\frac{1}{\prod_{u=1}^{n_{r+1}} \Gamma(s_{r+1,u})} \int_{0}^{\infty} \cdots \int_{0}^{\infty} t_{1}^{s_{r+1,1}-1} \cdots t_{n_{r+1}}^{s_{r+1, n_{r+1}}-1} \label{eq:QlinearLi}\\
&\times \frac{{\rm Li}_{\mathbf{k}}(e^{-(t_{1}+\dots+t_{n_{r+1}})}) }{e^{t_{1}+\dots+t_{n_{r+1}}} -1} \frac{1}{e^{t_{2}+\dots+t_{n_{r+1}}} -1} \cdots \frac{1}{e^{t_{n_{r+1}}} -1}\,dt_{1} \cdots dt_{n_{r+1}}. \nonumber
\end{align}
By applying (\ref{int}) with $r=2$ to the function (\ref{eq:QlinearLi}), we see that the function (\ref{eq:QlinearLi}) equals to $\zeta_{MT} (0, \mathbf{k}; \mathbf{s}_{r+1})$.
Moreover, by the definition of the function (\ref{eq:defU}), we find that 
\[\zeta_{MT} (0, \mathbf{k}; \mathbf{s}_{r+1}) = \zeta(\mathbf{k}, \mathbf{s}_{r+1}).\]
Therefore, (\ref{int}) is expressed as a $\mathbb{Q}$-linear combination of EZ-type MZFs.
\end{proof}

\section{Generalizing Theorem \ref{th:IZfrel} and Theorem \ref{th:AKZofMfre}}\label{se:Generalize}
In this section, using GMT-type MZFs, we generalize Theorem \ref{th:IZfrel} and Theorem \ref{th:AKZofMfre}.
\subsection{Ito zeta function}
It is difficult to generalize Theorem \ref{th:IZfrel} and Theorem \ref{th:AKZofMfre} as their original forms. First, in order to make the idea easier to understand, we rewrite Theorem \ref{th:IZfrel}.
\begin{proposition}\label{pr:Ito2funrelUtype}
For $r \in \mathbb{N},s \in \mathbb{C}$, we have
\begin{align*}
&\xi_{MT}(\{2\}^{r};s)\\
&=\sum_{a_{1}+a_{2}+a_{3}=r} \frac{r!}{a_{1}! a_{3}!} \zeta(2)^{a_{1}} (-1)^{a_{2} + a_{3}} \binom{s+a_{2}-1}{a_2}\zeta_{MT}(0,\{1\}^{a_{2}}, \{2\}^{a_{3}}; a_{2}+s),
\end{align*}
where the sum is over all $a_{1},a_{2},a_{3}\in \mathbb{Z}_{\ge 0}$ satisfying $a_{1}+a_{2}+a_{3}=r$.
\end{proposition}
\begin{proof}
By using the first formula of Corollary \ref{co:liint}, we have
\[{\rm Li}_{2}(1-e^{-t}) = \zeta(2) - \int_{0}^{\infty}\frac{t}{e^{t+u}-1}\,du-\int_{0}^{\infty}\frac{u}{e^{t+u}-1}\,du.\]
Therefore, for $\Re(s) > 1$, we have
\begin{align*}
&\Gamma(s)\xi_{MT}(\{2\}^{r};s)\\
&=\int_{0}^{\infty} \frac{t^{s-1}}{e^{t} - 1} \Bigl({\rm Li}_{2}(1-e^{-t})\Bigr)^{r}\,dt \\
&=\int_{0}^{\infty} \frac{t^{s - 1}}{e^{t} - 1} \Bigl(\zeta(2) - \int_{0}^{\infty}\frac{t}{e^{t+u}-1}\,du - \int_{0}^{\infty}\frac{u}{e^{t+u}-1}\,du\Bigr)^{r} \,dt \\
&=\sum_{a_{1}+a_{2}+a_{3}=r} \frac{r!}{a_{1}! a_{2}! a_{3}!}\\
&\quad\times \int_{0}^{\infty} \frac{t^{s - 1}}{e^{t} - 1} \zeta(2)^{a_{1}} \Bigl(-\int_{0}^{\infty}\frac{t}{e^{t+u}-1}\,du\Bigr)^{a_{2}} \Bigl(-\int_{0}^{\infty}\frac{u}{e^{t+u}-1}\,du\Bigr)^{a_{3}} \,dt \\ 
&=\sum_{a_{1}+a_{2}+a_{3}=r} \frac{r!}{a_{1}! a_{2}! a_{3}!} \zeta(2)^{a_{1}} (-1)^{a_{2} + a_{3}} \\
&\quad\times\int_{0}^{\infty} \frac{t^{s+a_{2} - 1}}{e^{t} - 1} \Bigl(\int_{0}^{\infty}\frac{1}{e^{t+u}-1}\,du\Bigr)^{a_{2}} \Bigl(\int_{0}^{\infty}\frac{u}{e^{t+u}-1}\,du\Bigr)^{a_{3}} \,dt \\ 
&=\sum_{a_{1}+a_{2}+a_{3}=r} \frac{r!}{a_{1}! a_{2}! a_{3}!} \zeta(2)^{a_{1}} (-1)^{a_{2} + a_{3}} \Gamma(a_{2}+s) \zeta_{MT}(0,\{1\}^{a_{2}}, \{2\}^{a_{3}}; a_{2}+s).
\end{align*}
By the analytic continuation of the Ito zeta function (\cite[Theorem 2]{Ito}) and the MT-type MZF, we obtain the stated theorem.
\end{proof}
The key of this proof is to use the formula
\[{\rm Li}_{2}(1-e^{-t}) = \zeta(2) - \int_{0}^{\infty}\frac{t}{e^{t+u}-1}\,du-\int_{0}^{\infty}\frac{u}{e^{t+u}-1}\,du\]
directly. We generalize this formula to any $k$.
\begin{lemma}\label{le:Liint}
For $k \in \mathbb{Z}_{\ge 2}$ and $t>0$, we have
\[{\rm Li}_{k}(1-e^{-t})=\sum_{j=0}^{2k-2} f(t; j, k),\]
where
\begin{align*}
  &f(t; j, k)  \\
&= \begin{cases}
    (-1)^{j} \zeta(k-j) \int_{0}^{\infty} \cdots \int_{0}^{\infty} \prod_{i=1}^{j}\frac{du_{i}}{e^{t+u_{1}+\dots +u_{i}} - 1}\hspace{\fill}({\rm if}\  j < k-1),\\
   (-1)^{k-1} \int_{0}^{\infty}  \cdots \int_{0}^{\infty} t \prod_{i=1}^{k-1}\frac{du_{i}}{e^{t+u_{1}+\dots +u_{i}}-1} \hspace{\fill}({\rm if}\  j = k-1), \\
  (-1)^{k-1}  \int_{0}^{\infty}  \cdots  \int_{0}^{\infty} u_{j-k+1}\prod_{i=1}^{k-1}\frac{du_{i}}{e^{t+u_{1}+\dots +u_{i}}-1} \hspace{\fill}({\rm if}\  j > k-1 ),
\end{cases}
\end{align*}
and
we understand if $\, j=0$ then
\begin{gather*}
\int_{0}^{\infty} \frac{1}{e^{t+u_{1}} - 1} \cdots \frac{1}{e^{t+u_{1}+\dots +u_{j}} - 1} \, du_{1} \cdots du_{j} =1.
\end{gather*}
\end{lemma}
\begin{proof}
We use induction. If $k=2$, it is true since 
\[{\rm Li}_{2}(1-e^{-t})=\zeta(2) - \int_{0}^{\infty}\frac{t}{e^{t+u}-1}\,du - \int_{0}^{\infty}\frac{u}{e^{t+u}-1}\,du.\]
Assume that the formula holds for $k$, and prove it for $k+1$.
\begin{align*}
&{\rm Li}_{k+1}(1-e^{-t}) \\
&=\zeta(k+1) - \int_{0}^{\infty}\frac{{\rm Li}_{k}(1-e^{-(t+u_{1})})}{e^{t+u_{1}} -1} \,du_{1} \\
&=\zeta(k+1) - \int_{0}^{\infty} \frac{1}{e^{t+u_{1}} -1} \Bigl(\sum_{j=0}^{2k-2} f(t+u_{1}; j, k)\Bigr) du_{1} \\
&=\zeta(k+1) + \sum_{j=0}^{k-2} f(t; j+1, k+1) + f(t; k, k+1)+ f(t; k+1, k+1) \\
&\quad+\sum_{j=k}^{2k-2} f(t; j+2, k+1)\\
&=\sum_{j=0}^{2k} f(t; j, k+1).
\end{align*}
This completes the proof.
\end{proof}
\begin{lemma}\label{le:Liprod}
With the assumption of Lemma \ref{le:Liint}, for $k_{i} \ge 2\, (1 \le i \le r)$ and $t>0$, we have
\[\prod_{i=1}^{r} {\rm Li}_{k_{i}}(1-e^{-t}) = \sum_{\substack{0 \le j_{1} \le 2k_{1}-2 \\ \vdots \\ 0 \le j_{r} \le 2k_{r}-2 }} \prod_{i=1}^{r} f(t; j_{i}, k_{i}).\]
\end{lemma}
\begin{proof}
By Lemma \ref{le:Liint}, we have
\[\prod_{i=1}^{r} {\rm Li}_{k_{i}}(1-e^{-t}) = \prod_{i=1}^{r} \sum_{j=0}^{2k_{i}-2} f(t; j, k_{i}) =\sum_{\substack{0 \le j_{1} \le 2k_{1}-2 \\ \vdots \\ 0 \le j_{r} \le 2k_{r}-2 }} \prod_{i=1}^{r} f(t; j_{i}, k_{i}).\]
\end{proof}
\begin{theorem}\label{th:ItoGfre}
For $l \in \mathbb{Z}_{\ge 0}$,\ $\,s \in \mathbb{C}$ and $\,k_{i} \ge 2\ (1 \le i \le r)$, we have
\begin{align*}
&\xi_{MT}(\{1\}^{l},k_{1},\dots ,k_{r};s) \\
&=\sum_{\substack{0 \le j_{1} \le 2k_{1}-2 \\ \vdots \\ 0 \le j_{r} \le 2k_{r}-2 }} a_{r}(\mathbf{j},\mathbf{k}) \frac{\Gamma(l+b_{r}(\mathbf{j},\mathbf{k})+s)}{\Gamma(s)}\\
&\quad\times\zeta_{MT}(0,\mathbf{k}(j_{1},k_{1}), \dots,\mathbf{k}(j_{r},k_{r});l+b_{r}(\mathbf{j},\mathbf{k})+s),
\end{align*}
where 
\[a_{r}(\mathbf{j},\mathbf{k}) = \prod_{i=1}^{r} a(j_{i},k_{i}),
\ \ \ \ \ a(j_{i}, k_{i}) = \begin{cases}
(-1)^{j_{i}}\zeta(k_{i}-j_{i}) & (j_{i} < k_{i}-1), \\
(-1)^{k_{i}-1} & (j_{i} \ge k_{i}-1),
\end{cases}\]
\[b_{r}(\mathbf{j},\mathbf{k}) = |\{ i \in \{1,\dots,r\} \mid j_{i} = k_{i}-1\}|,\]
and
\[\mathbf{k}(j_{i}, k_{i}) =  \begin{cases}
(\{1\}^{j_{i}}) & (j_{i} \le k_{i}-1), \\
(\underbrace{\{1\}^{j_{i}-k_{i}},2,\{1\}^{2k_{i}-2-j_{i}}}_{k_{i}-1}) & (j_{i} > k_{i}-1).
\end{cases}\]
\end{theorem}
\begin{remark*}
By putting $s=m+1$ in Theorem \ref{th:ItoGfre} and using Theorem \ref{th:IZspv} for the left hand side, we can obtain relations among GMT-type MZVs with $n_{r+1} = 1$.
\end{remark*}
\begin{proof}[Proof of Theorem \ref{th:ItoGfre}]
By lemma \ref{le:Liprod} and the formula (\ref{eq:Uintrep}), for $\Re(s) > 1$, we have
\begin{align*}
&\Gamma(s)\xi_{MT}(\{1\}^{l},k_{1},\dots ,k_{r};s)\\
&=\int_{0}^{\infty} \frac{t^{s+l-1}}{e^{t}-1}  \prod_{i=1}^{r} {\rm Li}_{k_{i}}(1-e^{-t})\,dt \\
&=\sum_{\substack{0 \le j_{1} \le 2k_{1}-2 \\ \vdots \\ 0 \le j_{r} \le 2k_{r}-2 }} \int_{0}^{\infty} \frac{t^{s+l-1}}{e^{t}-1}  \prod_{i=1}^{r} f(t; j_{i}, k_{i})\,dt \\
&=\sum_{\substack{0 \le j_{1} \le 2k_{1}-2 \\ \vdots \\ 0 \le j_{r} \le 2k_{r}-2 }} a_{r}(\mathbf{j},\mathbf{k}) \Gamma(l+b_{r}(\mathbf{j},\mathbf{k})+s) \\
&\quad\times\zeta_{MT}(0,\mathbf{k}(j_{1},k_{1}),\dots,\mathbf{k}(j_{r},k_{r});l+b_{r}(\mathbf{j},\mathbf{k})+s).
\end{align*}
By the analytic continuation, we obtain the stated theorem.
\end{proof}
\subsection{An analog of the Arakawa-Kaneko zeta function of Miyagawa-type}
We generalize Theorem \ref{th:AKZofMfre} by using  Lemma \ref{le:Liprod}.
\begin{theorem}\label{th:AKZofMGfre}
With the assumption of Theorem \ref{th:ItoGfre}, for $l \in \mathbb{Z}_{\ge 0}$, $\mathbf{k}=(k_{1},\dots,k_{n})\in\mathbb{Z}_{\ge2}^{n}$, $k_{n+1} \in \mathbb{Z}_{\ge0}$ and $s \in \mathbb{C}$, we have
\begin{align*}
&(-1)^{k_{n+1}}\xi_{MT}((\{1\}^{l},\mathbf{k};k_{n+1});s) \\
&=\sum_{\substack{0 \le j_{1} \le 2k_{1}-2 \\ \vdots \\ 0 \le j_{n} \le 2k_{n}-2 }} \sum_{c_{1}+\dots+c_{k_{n+1}+1}=l+b_{n}(\mathbf{j},\mathbf{k})}  \bigl(l+b_{n}(\mathbf{j},\mathbf{k})\bigr)! a_{n}(\mathbf{j},\mathbf{k}) \binom{s+c_{k_{n+1}+1}-1}{c_{k_{n+1}+1}} \\
&\quad\times \zeta_{MT}(0,\mathbf{k}(j_{1},k_{1}),\dots,\mathbf{k}(j_{n},k_{n});c_{1}+1,\dots,c_{k_{n+1}}+1,c_{k_{n+1}+1}+s)\\
&\quad-\sum_{i=1}^{k_{n+1}}(-1)^{i-1} \zeta_{MT}(\{1\}^{l},\mathbf{k};i) \zeta(\{1\}^{k_{n+1}-i},s).
\end{align*}
\end{theorem}
\begin{remark}
By putting $s=m+1$ in Theorem \ref{th:AKZofMGfre} and using Theorem \ref{th:AKZofMspv} for the left hand side, we can obtain relations among GMT-type MZVs.
\end{remark}
\begin{proof}
Let $q=k_{n+1}+1$. For $\Re(s) > 1$, let 
\[J=\int_{0}^{\infty} \cdots \int_{0}^{\infty} t_{q}^{s-1} (t_{1}+\dots+t_{q})^{l} \left(\prod_{i=1}^{n}{\rm Li}_{k_{i}}(1 - e^{-(t_{1}+\dots+t_{q})})\right) \prod_{i=1}^{q}\frac{dt_{i}}{e^{t_{i}+\dots+t_{q}}-1}.\]
We calculate $J$ in two different ways.
By using Corollary \ref{co:liint}, we have
\begin{align*}
J&=\int_{0}^{\infty} \cdots \int_{0}^{\infty} t_{q}^{s-1} \Bigl(\zeta_{MT}(\{1\}^{l},\mathbf{k};1) - {\rm Li}_{\{1\}^{l},\mathbf{k};1}(1 - e^{-(t_{2}+\dots+t_{q})})\Bigr) \\
&\quad\times \frac{1}{e^{t_{2}+\dots+t_{q}}-1}  \cdots \frac{1}{e^{t_{q}}-1} \,dt_{2}\cdots dt_{q} \\
&=\Gamma(s) \zeta_{MT}(\{1\}^{l},\mathbf{k};1) \zeta(\{1\}^{q-2},s) \\
&\quad-\int_{0}^{\infty} \cdots \int_{0}^{\infty} t_{q}^{s-1} {\rm Li}_{\{1\}^{l},\mathbf{k};1}(1 - e^{-(t_{2}+\dots+t_{q})})\\
&\quad\times \frac{1}{e^{t_{2}+\dots+t_{q}}-1}  \cdots   \frac{1}{e^{t_{q}}-1} \,dt_{2}\cdots dt_{q} \\
&=\cdots \\
&=\Gamma(s) \sum_{i=1}^{q-1}(-1)^{i-1} \zeta_{MT}(\{1\}^{l},\mathbf{k};i) \zeta(\{1\}^{q-1-i},s) \\
&\quad+(-1)^{q-1} \Gamma(s) \xi_{MT}((\{1\}^{l},\mathbf{k};q-1);s)\, .
\end{align*}

On the other hand, by using Lemma \ref{le:Liprod}, we have

\begin{align*}
J&=\sum_{\substack{0 \le j_{1} \le 2k_{1}-2 \\ \vdots \\ 0 \le j_{n} \le 2k_{n}-2 }} \int_{0}^{\infty} \cdots \int_{0}^{\infty} t_{q}^{s-1} ( t_{1}+\dots+t_{q} )^{l}\\
&\quad\times\left(\prod_{i=1}^{n} f(t_{1}+\dots+t_{q}; j_{i}, k_{i})\right) \prod_{i=1}^{q}\frac{dt_{i}}{e^{t_{i}+\dots+t_{q}}-1}\\
&=\sum_{\substack{0 \le j_{1} \le 2k_{1}-2 \\ \vdots \\ 0 \le j_{n} \le 2k_{n}-2 }} \int_{0}^{\infty} \cdots \int_{0}^{\infty} t_{q}^{s-1} ( t_{1}+\dots+t_{q} )^{l+b_{n}(\mathbf{j},\mathbf{k})} \\
&\quad\times \left(\prod_{i=1}^{n} \frac{f(t_{1}+\dots+t_{q}; j_{i}, k_{i})}{( t_{1}+\dots+t_{q} )^{b_{n}(\mathbf{j},\mathbf{k})}}\right) \prod_{i=1}^{q}\frac{dt_{i}}{e^{t_{i}+\dots+t_{q}}-1}\\
&=\sum_{\substack{0 \le j_{1} \le 2k_{1}-2 \\ \vdots \\ 0 \le j_{n} \le 2k_{n}-2 }} \sum_{c_{1}+\dots+c_{q}=l+b_{n}(\mathbf{j},\mathbf{k})}  \frac{\bigl(l+b_{n}(\mathbf{j},\mathbf{k})\bigr)!}{c_{1}!\cdots c_{q}!} \int_{0}^{\infty} \cdots \int_{0}^{\infty} t_{q}^{s-1} t_{1}^{c_{1}} \cdots t_{q}^{c_{q}} \\
&\quad\times \left(\prod_{i=1}^{n} \frac{f(t_{1}+\dots+t_{q}; j_{i}, k_{i})}{( t_{1}+\dots+t_{q} )^{b_{n}(\mathbf{j},\mathbf{k})}}\right) \prod_{i=1}^{q}\frac{dt_{i}}{e^{t_{i}+\dots+t_{q}}-1}\\
&=\sum_{\substack{0 \le j_{1} \le 2k_{1}-2 \\ \vdots \\ 0 \le j_{n} \le 2k_{n}-2 }} \sum_{c_{1}+\dots+c_{q}=l+b_{n}(\mathbf{j},\mathbf{k})}  \frac{\bigl(l+b_{n}(\mathbf{j},\mathbf{k})\bigr)!}{c_{1}!\cdots c_{q}!} a_{n}(\mathbf{j},\mathbf{k})\left(\prod_{i=1}^{q-1} \Gamma(c_{i}+1) \right)\\
&\quad\times\Gamma(c_{q}+s) \zeta_{MT}(0,\mathbf{k}(j_{1},k_{1}),\dots,\mathbf{k}(j_{n},k_{n});c_{1}+1,\dots,c_{q-1}+1,c_{q}+s)\\
&=\sum_{\substack{0 \le j_{1} \le 2k_{1}-2 \\ \vdots \\ 0 \le j_{n} \le 2k_{n}-2 }} \sum_{c_{1}+\dots+c_{q}=l+b_{n}(\mathbf{j},\mathbf{k})}  \frac{\bigl(l+b_{n}(\mathbf{j},\mathbf{k})\bigr)!}{c_{q}!} a_{n}(\mathbf{j},\mathbf{k}) \Gamma(c_{q}+s) \\
&\quad\times \zeta_{MT}(0,\mathbf{k}(j_{1},k_{1}),\dots,\mathbf{k}(j_{n},k_{n});c_{1}+1,\dots,c_{q-1}+1,c_{q}+s).
\end{align*}
By the analytic continuation, we obtain the stated theorem.
\end{proof}
\section*{Acknowledgment}
The author is deeply grateful to Prof. Kohji Matsumoto, Mr. Tomohiro Ikkai, Mr. Yuta Suzuki, Mr. Kenta Endo for their helpful comments.

\vspace{4mm}

{\footnotesize
{\sc
\noindent
Graduate School of Mathmatics, Nagoya University,\\
Chikusa-ku, Nagoya 464-8602, Japan.
}\\
{\it E-mail address}, R. Umezawa\hspace{1.75mm}: {\tt
m15016w@math.nagoya-u.ac.jp}\\
}
\end{document}